\documentclass[12pt]{amsart}

\usepackage[mathscr]{eucal}
\usepackage{amssymb,amsmath}
\usepackage{amsthm}
\usepackage{verbatim,multicol}
\usepackage{enumerate} 
\usepackage{hyperref}
\usepackage[msc-links]{amsrefs}

\renewcommand{\qedsymbol}{\hfill{\vrule height7pt width6pt depth-1pt}}

\newcommand{\bbA}{\mathbb{A}}
\newcommand{\bbN}{\mathbb{N}}
\newcommand{\bbZ}{\mathbb{Z}}

\newcommand{\cO}{\mathcal{O}}

\newcommand{\fkm}{\mathfrak{m}}

\newcommand{\End}{\operatorname{End}}
\newcommand{\Aut}{\operatorname{Aut}}

\newcommand{\ann}{\operatorname{ann}}
\newcommand{\id}{\operatorname{id}}

\newcommand{\Inn}{\operatorname{Inn}}

\newcommand{\define}[1]{\textbf{#1\index{#1}}}

\newcommand{\iso}{\cong}

\newcommand{\sub}{\subseteq}
\newcommand{\notsub}{\not\subseteq}

\newcommand{\normal}{\triangleleft}
\newcommand{\compose}{\circ}
\newcommand{\cp}{\compose}

\newcommand{\Th}{^{\text{th}}}

\theoremstyle{plain}
\newtheorem{thm}{Theorem}[section]
\newtheorem{lem}[thm]{Lemma}
\newtheorem{cor}[thm]{Corollary}
\newtheorem{prop}[thm]{Proposition}

\theoremstyle{definition}

\newtheorem{ex}[thm]{Example}

\theoremstyle{plain}

\theoremstyle{definition}

\newtheorem*{remx}{Remark}

\newcommand{\ssa}[2]{{#2}^{#1}}
\newcommand{\ssaa}[3]{({#3}^{#2})^{#1}}
\newcommand{\sss}{$\sigma$-stable}
\newcommand{\GGs}{$G$-stable}

\newcommand{\ssi}{$\sigma$-invariant}
\newcommand{\GGi}{$G$-invariant}

\newcommand{\hasacc}{satisfies the a.c.c.\ on ideals}

\newcommand{\psurj}{$\sigma$ is pointwise surjective}

\newcommand{\revsh}{the reversibility condition}
\newcommand{\Revsh}{The reversibility condition}
\newcommand{\rev}{$(G,\sigma)$ is reversible up to conjugation}

\newcommand{\abs}[1]{\lvert #1\rvert}
\newcommand{\Ass}{\operatorname{Ass}}
\newcommand{\st}[2]{(#1:#2)}

\newcommand{\smin}{\setminus}
\newcommand{\smid}{\mid}

\newcommand{\upm}{unique product monoid}
\newcommand{\upg}{unique product group}

\newcommand{\skewm}[2]{#1*#2}

\newcommand{\basis}[1]{\mspace{1mu}\overline{#1\mspace{2mu}}\mspace{1mu}}

\begin{document}

\title[Associated primes of induced modules]%
{Prime Modules and Associated Primes of Induced Modules over Rings
Graded by Unique Product Monoids}

\author{Allen D. Bell}

\address{Dept.~of Mathematical Sciences, University of Wisconsin--Milwaukee}

\email{adbell@uwm.edu}



\begin{abstract}
We study prime ideals, prime modules, and associated primes of graded
modules over rings $S$ graded by a \upm. We consider two situations in
detail: (a) the case where $S$ is strongly group-graded and (b) the case
where $S$ is a crossed product and the ideal or module is induced from the
identity component $R$ of $S$. We give explicit conditions for ideals and
modules of $R$ to induce prime ideals of or prime modules over $S$ in
these two cases. We then describe the set of associated prime ideals of an
arbitrary induced module.

One of our main interests is to give necessary and sufficient conditions for
primeness, and to describe the associated primes, in the crossed product
case when the action of the monoid is not an action by automorphisms; this
includes the case of a skew polynomial ring $R[x;\sigma]$ where $\sigma$
is an endomorphism of $R$.

At the end, we give some illustrative examples, several of which show the
necessity of the various hypotheses in our results.
\end{abstract}

\maketitle

\section{Introduction}\label{Introduction}

Suppose $S$ is a ring graded by a \upm{} $G$ and $R=S_e$ is the
identity component of $S$.  Suppose also that $I$ is an ideal of $R$ and
$M$ is a right $R$-module, whence the induced right $S$-module
$M\otimes_R S$ is a graded right $S$-module.  In each of
Sections~\ref{The strongly graded case}, \ref{The crossed product case},
and \ref{The skew polynomial and skew laurent cases}, we give necessary
and sufficient conditions for $IS$ to be a prime ideal of $S$, and, more
generally, for the induced $S$-module $M\otimes_R S$ to be prime. We
then give a description of the annihilators of the prime submodules of
$M\otimes_R S$; these are known as the associated primes of
$M\otimes_R S$.

The definitions and general results used in the rest of the paper are
introduced in Section~\ref{Definitions and basic results}.  Assuming the
grading monoid to be a \upm{} allows us to extend arguments used in
studying $\bbN$- and $\bbZ$-graded rings, such as degree arguments.
We are able to conclude that the associated primes of a graded $S$-module
must be homogeneous.

In Section~\ref{The strongly graded case}, we consider the case where
$G$ is a group and the grading is strong.  In this case, there is a natural
notion of \GGi{} ideal of $R$, and we show that $IS$ is a prime ideal if and
only if $I$ is a \GGi{} ideal such that whenever $I$ contains a product of
\GGi{} ideals, $I$ contains one of them.  Let $\st{I}{G}$ denote the
largest \GGi{} ideal contained in the ideal $I$ of $R$.  We show that a
graded $S$-module is prime if and only if $\st{\ann N}{G}=\st{\ann
M}{G}$ for all nonzero $R$-submodules $N$ of the identity component
$M$ of the module.  We then show that the associated primes of a graded
module are of the form $KS$ where $K$ is the largest \GGi{} ideal
contained in the annihilator of an $S$-prime submodule of $M$.

In Section~\ref{The crossed product case}, we study the case where $S$
is a crossed product over $R$ with an action $r\mapsto
\ssa{g}{r}$ of $G$ on $R$. When $G$ is not a group and the action of $G$ on $R$
is not by automorphisms, the situation is naturally more
complicated. For example, we must distinguish between
\GGs{} and \GGi{} ideals. We give necessary and sufficient conditions for
$IS$ or $M\otimes_R S$ to be prime, but the conditions are less
symmetric than in Section~\ref{The strongly graded case}.  For example,
Proposition~\ref{S-prime module for R*G} shows that $M\otimes_R S$ is
a prime $S$-module if and only if for any $m\in M, b\in R, g\in G$, if
$m\ssa{g}{R}\ssaa{g}{h}{b}=0$ for all $h\in G$, then $m=0$ or
$Mb=0$. Sometimes nicer conditions can be given for primeness  with an
extra hypothesis.  For example, Corollary~\ref{S-prime ideal R*G comm
or surj} shows that when each map $r\mapsto \ssa{g}{r}$ is surjective, a
\GGi{} ideal $I$ of $R$ induces a prime ideal $IS$ if and only if whenever
$A,B$ are ideals of $R$ with $B$ \GGs{} and $AB\sub I$, we have
$A\sub I$ or $B\sub I$. This result is also true without the surjectivity
hypothesis if $R$ is commutative, but another hypothesis about
reversibility of the action of $G$ must then be added.   This reversibility
hypothesis is a novel feature of considering non-commutative monoids that
are not groups; in particular, it has no analog in the study of skew
polynomial rings.  When $R$ \hasacc{} and $G$ acts by automorphisms,
conditions just like those in the strongly group-graded case are necessary
and sufficient.

We also describe the associated primes of $M\otimes_R S$ as extensions
$KS$ where $K$ is the largest \GGs{} ideal contained in the annihilator of
an $S$-prime submodule of some twist of $M$.  These twists are shown to
be necessary in general by examples in Section~\ref{Examples}, but may
be eliminated when each map $r\mapsto \ssa{g}{r}$ is surjective and the
reversibility hypothesis mentioned in the previous paragraph holds.

We discuss skew polynomial and skew laurent rings in Section~\ref{The
skew polynomial and skew laurent cases}; the results are special cases of
the results in previous sections but often have a simpler form.  One of the
motivations for this paper was to understand conditions for the skew
polynomial ring $R[x;\sigma]$ to be prime and generalize them to
characterize associated primes.  Various definitions of a $\sigma$-prime
ring have been given over the years: see for example, Goldie-Michler
\cite{GM74}, Pearson--Stephenson \cite{PS77}, and Irving \cite{Irv79}.
(See the discussion in Section~\ref{The skew polynomial and skew laurent
cases} for more details.)  All of these definitions give necessary and
sufficient conditions for $R[x;\sigma]$ to be prime in special cases, but we
are unaware of any published definition that works in complete generality.
We give such a definition here.

Even if $\sigma$ is an automorphism, the skew laurent ring $R[x^{\pm1
};\sigma]$ can be prime when the skew polynomial ring $R[x;\sigma]$ is
not.  This indicates that we should refer to $S$-primeness rather than
$\sigma$-primeness, where $S=R[x;\sigma]$ or
$S=R[x^{\pm1};\sigma]$ is the extension ring we are concerned with.
That is the approach we take in this paper.

As noted above, Section~\ref{Examples} collects various examples
illustrating the results of the paper and showing that various hypotheses
are necessary.  Example~\ref{affine example} gives a good picture of our
results in a simple setting. It computes the associated primes in the skew
polynomial case when $R$ is the coordinate ring of an affine algebraic set
and $M$ is the simple module corresponding to a point. The example is
still complex enough to exhibit phenomena that do not occur for
automorphic skew polynomial rings.

\section{Definitions and basic results}\label{Definitions and basic results}

Throughout this paper $R$ will be a subring of $S$.  We are interested in
the connection between primeness of ideals and $R$-modules induced
from $R$ to $S$.  We give some general definitions and then quickly
specialize to the case we will study: $S$ is a $G$-graded ring for a \upm{}
$G$ and $R$ is the identity component of $S$.  In this section, we prove
that such rings and induced modules possess properties that are similar to
properties of polynomial extensions.  In particular, we show that it is
generally enough to consider graded ideals and graded submodules.

\medskip

A nonzero $R$-module $M$ is said to be \define{prime} if $\ann N=\ann
M$ for all nonzero submodules $N$.  If $M$ is any $R$-module, an
\define{associated ideal} of $M$ is the annihilator of a prime submodule of
$M$.  We denote the set of all associated ideals of $M$ by $\Ass M$. The
annihilator of a prime module is always a prime ideal, and so an associated
ideal is also called an \define{associated prime}. Conversely, an ideal $I$ is
prime if and only if the module $R/I$ is prime.

Note that $M$ is prime if and only if for any $m\in M$ and any ideal
$J\normal R$ with $mJ=0$, we have either $m=0$ or $MJ=0$.  Another
equivalent condition for $M$ to be prime is that whenever $m\in M$,
$b\in R$ and $mRb=0$, either $m=0$ or $Mb=0$.  In each of
Sections~\ref{The strongly graded case}, \ref{The crossed product case},
and \ref{The skew polynomial and skew laurent cases}, we generalize
these conditions to equivalent conditions for the primeness of the induced
module $M\otimes_R S$ in the setting where $M$ is an $R$-module and
$S$ is an overring of $R$.

\smallskip

We say an ideal $I$ of $R$ is \define{right $S$-stable} if $SI\sub IS$, that
is, if $IS$ is an ideal of $S$. We say $I$ is \define{$S$-stable} if $IS=SI$.
We say an ideal $I$ of $R$ is \define{right $S$-prime} if $IS$ is a prime
ideal of $S$.  Such an $I$ must be right $S$-stable. We say $R$ is an
$S$-prime ring if $0$ is a right $S$-prime ideal, i.e., if $S$ is a prime ring.

A similar definition of $S$-stable was given in Montgomery-Schneider
\cite{MS99}, although the name they used would be ``$G$-stable'' in our
case.  They did not make the analogous definition for $S$-prime.

If $M$ is a right $R$-module, we say $M$ is \define{$S$-prime} if the
induced $S$-module $M\otimes_R S$ is prime.

\begin{lem}\label{S-prime ideal vs S-prime module}
Let $I$ be a right $S$-stable ideal of $R$.  Then the right $R$-module
$R/I$ is $S$-prime if and only if the ideal $I$ is right $S$-prime.
\end{lem}

\begin{proof}
Since $IS\normal S$, our remarks after the definition of prime module
imply $S/IS$ is a prime module if and only if $IS$ is a prime ideal.  Since
$R/I\otimes_R S\iso S/IS$ as right $S$-modules, this proves the lemma.
\end{proof}

\bigskip

A \define{\upm} is a monoid $G$ with the property that whenever $X,Y$
are nonempty finite subsets of $G$, there exist $x\in X,y\in Y$ such that
$xy\ne x'y'$ for any pair $(x',y')\in X\times Y$ different from $(x,y)$.  By
choosing the sets $X,Y$ appropriately, we see that a \upm{} must be
cancellative and torsionfree (i.e., if $a$ is not the identity, all powers
$a,a^2,\dots$ are distinct). If a \upm{} is a group, we call it a
\define{\upg}.

\begin{remx}
Any orderable monoid is a \upm.  Thus, for example, any submonoid of a
free group or a torsionfree nilpotent group is a \upm.

The group generated by $x,y$ subject to the relations
$x^{-1}y^2x=y^{-2}$ and $y^{-1}x^2y=x^{-2}$ is torsionfree but not a
\upg.  For a discussion of this and related examples, see Carter \cite{Car13}.

The monoid generated by $x,y$ subject to the relations $xy=yx$ and
$x^2=y^2$ is commutative, torsionfree, and cancellative but is not a
\upm, as can be seen by taking $X=Y=\{\,x,y\,\}$.
\end{remx}

\noindent\textbf{Conventions.}
Throughout this paper, $G$ will be a \upm{} or a \upg{} with identity $e$.
We will use $\bbN$ to denote the additive monoid of nonnegative integers.
Modules will be unital right modules and ideals will be two-sided unless the
contrary is explicitly stated.

\medskip

Recall that an abelian group $A$ is \define{$G$-graded} if a
decomposition $A=\oplus_{g\in G} A_g$ into a direct sum of subgroups is
given. The nonzero elements of $A_g$ are said to be
\define{homogeneous} of \define{degree} $g$.  For a homogeneous $a\in
A$, we write $\partial a$ for the degree of $a$.   Note that $0$ is not
regarded as a homogeneous element, although we will sometimes write
``nonzero homogeneous element'' for emphasis.

If $a\in A$ is arbitrary, we say an expression $a=a_1+\dots+a_k$ is a
\define{canonical form} for $a$ if the elements $a_i$ are homogeneous
and their degrees $\partial a_i$ are distinct.  If $a=a_1+\dots+a_k$ is a
canonical form, the $\partial a_i$ \define{homogeneous component} of
$a$ is $a_i$. We call $\{\,\partial a_1,\dots,\partial a_k\,\}$ the
\define{support} of $a$. A subset $B\sub A$ is said to be
\define{homogeneous} if the components of any element of $B$ are all in
$B$.

If $A$ is a ring, we say it is a \define{$G$-graded} ring if $A_gA_h\sub
A_{gh}$ for all $g,h\in G$ and $1_A\in A_{e}$.  If $A$ is a $G$-graded
ring and $M=\oplus_{g\in G}M_g$ is a right $A$-module that is
$G$-graded, we say $M$ is a \define{graded module} if $M_gA_h\sub
M_{gh}$ for all $g,h\in G$.

\smallskip

Since $G$ is a cancellation monoid, degrees and homogeneous components
are fairly well-behaved.  This is exhibited in the following lemma and
corollary. Recall that if $X$ is a subset of $R$ or of an $R$-module $M$,
then $\ann X=\{\,r\in R\smid xr=0\text{ for all }x\in X\,\}$.  We leave
the proofs to the reader.

\begin{lem}\label{degree works}
Let $G$ be a cancellation monoid,  $S$ a $G$-graded ring, and $M$ a
graded right $S$-module.
\begin{enumerate}
\item
If $m\in M$  is homogeneous and $s=s_1+\dots+s_k\in A$ is a
canonical form, then $ms=ms_1+\dots+ms_{k}$ becomes a canonical
form when we omit any $ms_i$ that are $0$.  In particular, $ms=0$ if
and only if each $ms_i=0$.
\item
If $s\in S$  is homogeneous and $m=m_1+\dots+m_k\in M$ is a
canonical form, then $ms=m_1s+\dots+m_ks$ becomes a canonical
form when we omit any $m_is$ that are $0$.  In particular, $ms=0$ if
and only if each $m_is=0$.
\qedsymbol
\end{enumerate}
\end{lem}

\begin{cor}\label{homog annihilator}
Let $G$ be a cancellation monoid,  $S$ a $G$-graded ring, and $M$ a
graded right $S$-module.
\begin{enumerate}
\item
If $m\in M$ is homogeneous, then $\ann m$ is a homogeneous right
ideal of $S$.
\item
$\ann M$ is a homogeneous ideal of $S$.
\qedsymbol
\end{enumerate}
\end{cor}

\medskip

We now prove a general result about submodules of graded modules over
rings graded by a \upm. This pivotal result  is presumably rather old in the
cases $G=\bbN$ and $G=\bbZ$.  The idea and the proof often show up in
results like our Corollary~\ref{anns of prime submods of homog mods}
below (often without our preliminary lemma); see the comments before
that corollary for more citations. A similar idea was used in the paper
Jespers-Krempa-Puczy{\l}owski \cite{JKP82} to study radicals in rings
graded by \upm{}s. The result is loosely related to the notion of ``good''
polynomial employed in Leroy-Matczuk \cite{LM04} to study associated
primes for $R[x;\sigma,\delta]$.

\begin{lem}\label{anns of submods of homog mods}
Let $G$ be a \upm, $S$ a $G$-graded ring, and $M$ a graded right
$S$-module.  Suppose $N$ is a nonzero submodule of $M$ and $a\in N$
has canonical form $a=a_1+\dots+a_k$ with $k$ minimal for a nonzero
element of $N$. Then the following statements hold.
\begin{enumerate}
\item
$\ann a=\ann a_i$ for $i=1,\dots,k$.
\item
$\ann a$ is homogeneous.
\item
$\ann aS$ is homogeneous.
\end{enumerate}
\end{lem}

\begin{proof}
(1) Let $s=s_1+\dots+s_{\ell}$ be a canonical form for a nonzero $s\in S$
satisfying $as=0$. We will show by induction on $\ell$ that $a_is_j=0$ for
all $i,j$. First suppose $\ell=1$, i.e., $s$ is homogenous.  By
Lemma~\ref{degree works}(2), $a_is=0$ for all $i$.

Now suppose $\ell>1$.  By the unique product property, there exist
subscripts $m,n$ such that the product $\partial a_m \partial s_n$ is
unique. Since $as=0$, the product $a_ms_n$ must equal $0$.  Thus
$as_n$ has smaller support than $a$.  By the minimality of $k$, this forces
$as_n=0$, which in turn implies $a_is_n=0$ for all $i$.  This also implies
$a(s-s_n)=0$, so by induction on $\ell$, we have that $a_is_j=0$ for all
$i$ and all $j\ne n$. This finishes the proof of (1).

(2) This is an immediate consequence of (1) and Corollary~\ref{homog
annihilator}(1).

(3) This follows from (2) and Corollary~\ref{homog annihilator}(2).
\end{proof}

Recall that when $G$ is a group, we say a $G$-graded ring $S$ is
\define{strongly graded} if $S_gS_h=S_{gh}$ for all $g,h\in G$.

\begin{cor}\label{submods of homog mods}
Let $G$ be a \upm, $S$ a $G$-graded ring, $M$ a graded right
$S$-module, and $N$ a nonzero submodule of $M$.  Then $N$ contains an
isomorphic copy of $bS$ for some (nonzero) homogeneous $b\in M$.  If
$G$ is a group and $S$ is strongly graded, then we can take $b$ to have
degree $e$.
\end{cor}

\begin{proof}
Let $a\in N$ be as in Lemma~\ref{anns of submods of homog mods}, with
canonical form $a=a_1+\dots+a_k$. Then $a_1\in M$ and $aS\sub N$ is
isomorphic to the cyclic module $S/\ann a=S/\ann a_1\iso a_1S$. Thus
we may take $b=a_1$.

If $S$ is strongly graded and $a$ has a nonzero component in degree $g\in
G$, then $aS_{g^{-1}}$ contains an element $a'$ with nonzero component
$b$ in degree $e$.  This $a'$ has support no larger than that of $a$ and is
still in $N$. By Lemma~\ref{anns of submods of homog mods}, we see
that $\ann a'=\ann b$.  Thus $bS\iso a'S\sub N$.
\end{proof}

\smallskip

The following corollary of Lemma~\ref{anns of submods of homog mods}
tells us that the associated prime ideals of a graded module are all
homogeneous and can be found by considering only graded submodules.
This result for $G=\bbN$ (and $R$ commutative) is Proposition 3.12 in
Eisenburd \cite{Eis95}.  Exercise 3.5 in \cite{Eis95} asks the reader to
extend the result to gradings by ordered abelian monoids. The corollary is
also a key result in the study of associated primes of induced modules for
$R[x;\sigma]$ in Nordstrom \cite[Proposition~3.1]{Nor05}.

\begin{cor}\label{anns of prime submods of homog mods}
Let $G$ be a \upm, $S$ a $G$-graded ring, $M$ a graded right
$S$-module, and $N$ a prime submodule of $M$.  Then $\ann N$ is a
homogenous prime ideal and is equal to $\ann N'$ for some prime, graded
submodule $N'$ of $M$.
\end{cor}

\begin{proof}
This is immediate from Lemma~\ref{anns of submods of homog mods}
and Corollary~\ref{submods of homog mods}.
\end{proof}

\smallskip

The next result is well-known in the cases $G=\bbN$ and $G=\bbZ$.  See
for example, N{\u{a}}st{\u{a}}sescu-van Oystaeyen
\cite[Prop.~II.1.4]{NvO82}.

\begin{lem}\label{graded prime}
Let $G$ be a \upm{} and let $S$ be a $G$-graded ring. Let $I$ be a
homogenous ideal of $S$ and let $M$ be a graded $S$-module.
\begin{enumerate}
\item
$I$ is prime if and only if for any homogeneous $a,b\in S$, if
$aS_gb\sub I$ for all $g\in G$, then $a\in I$ or $b\in I$.
\item
$M$ is prime if and only if for any homogeneous $m\in M,b\in S$, if
$mS_gb=0$ for all $g\in G$, then $m=0$ or $Mb=0$.
\end{enumerate}
\end{lem}

\begin{proof}
(1)  This follows from (2) if we set $M=S/I$.

(2) If $M$ is prime and $mS_gb=0$ for all $g\in G$, then $mSb=0$.
 Thus $m=0$ or $b\in\ann M$.  This proves the ``only if'' direction.

Suppose now that $M$ satisfies the ``if'' hypothesis and suppose
$mSb=0$ for nonzero $m\in M$, $b\in S$.  Set $N=\{\,m'\in M\smid
m'Sb=0\,\}$; this $N$ is a nonzero submodule of $M$.  By
Corollary~\ref{submods of homog mods}, there is a homogeneous element
$x\in M$ such that $xS$ is isomorphic to a submodule of $N$.  In
particular, $xSb=0$.

Let $b=b_1+\dots+b_{\ell}$ be a canonical form.  Since $\ann xS$ is
homogenous and $b\in \ann xS$, we have $xSb_j=0$ for each $j$. By
hypothesis, this implies $Mb_j=0$ for each $j$, and so $Mb=0$.
\end{proof}

\section{The strongly graded case}\label{The strongly graded case}

In this section, we assume $G$ is a \upg, $S$ is a strongly $G$-graded
ring, and $R=S_e$. We give alternative characterizations of the notion of
$S$-prime in this case. For example, we show $IS$ is a prime ideal of $S$
for an $S$-stable ideal $I$ of $R$ if and only if whenever $J,K$ are
$S$-stable ideals of $R$ with $JK\sub I$, we have $J\sub I$ or $K\sub
I$.  This result --- in the case $I=0$, but this limitation is easily removed
--- was proved for any torsionfree group $G$ by
Passman \cite[Corollary 4.6]{Pas84} or \cite[Corollary 8.5]{Pas89}, using
more complicated methods.

We also show that the associated primes of a graded $S$-module $N$ are
precisely the ideals $IS$ where $I$ is the largest $S$-stable ideal
contained in the annihilator of some $S$-prime submodule of the
$R$-module $N_e$ of degree $e$ elements of $N$.

\smallskip

If $S$ is strongly $G$-graded and $N$ is a graded $S$-module, then the
multiplication map from $N_g\otimes_R S_h$ to $N_{gh}$ is an
$R$-module isomorphism for all $g,h\in G$, and the multiplication map
from $N_e\otimes_R S$ to $N$ is a graded $S$-module isomorphism.
(See Dade \cite[Theorem~2.8]{Dad80} and N{\u{a}}st{\u{a}}sescu-van
Oystaeyen \cite[Theorem~1.3.4 \& Proposition~1.3.6]{NvO82}.) In
particular, $N_gS_h=N_{gh}$ for all $g,h\in G$.  If $M$ is an
$R$-module, this isomorphism allows us to identify $M\otimes_R S_g$
with the formal product $MS_g$.

If $I\normal R$, $S$ is strongly $G$-graded, and $g\in G$, we define
$\ssa{g}{I}=S_gIS_{g^{-1}}$. This is again an ideal of $R$, and it is easy
to see that $\ssaa{g}{h}{I}=\ssa{gh}{I}$.  Thus the map taking $g$ to
$I\mapsto \ssa{g}{I}$  is a homomorphism from $G$ to the
automorphism group of the lattice of ideals of $R$.

Some authors define $\ssa{g}{I}=S_{g^{-1}}IS_g$ so that
$\ssaa{h}{g}{I}=\ssa{gh}{I}$, but our definition is consistent with the use
of the notation in the rest of this paper, where we are required to define
$\ssa{g}{I}$ without the existence of $g^{-1}$.

We say an ideal $I$ of $R$ is \define{\GGi} if $\ssa{g}{I}=I$ for all $g\in
G$. Since $G$ is a group, it is easy to see that this is equivalent to the
assumption that $\ssa{g}{I}\sub I$ for all $g\in G$.

The next result relates $S$-stability and $G$-invariance, and it shows
that in the strongly graded case, we can drop the adjective ``right''.

\begin{lem}\label{sg-good conditions}
Let $G$ be a group and suppose $S$ is a strongly $G$-graded ring with
$R=S_e$. If $I$ is an ideal of $R$, then the following statements are
equivalent.
\begin{enumerate}
\item
$I$ is right $S$-stable.
\item
$IS=SI$.
\item
$I$ is \GGi.
\end{enumerate}
\end{lem}

\begin{proof}
$(1)\iff (3)$ Clearly $IS\normal S$ if and only if $SI\sub IS$. This
containment holds if and only if $S_gI\sub IS_g$ for all $g\in G$, which is
the case if and only if $\ssa{g}{I}\sub I$ for all $g\in G$.  By our remark
before the lemma, this is equivalent to the statement that $I$ is
\GGi.

$(3)\iff (2)$  This follows from $(3)\implies(1)$ and the ``left'' version of
$(3)\iff (1)$.
\end{proof}

If $I\normal R$, we define $\st{I}{G}=\cap_{g\in G}\ssa{g}{I}$.  It is
easy to see that $\st{I}{G}$ is the largest \GGi{} ideal of $R$ contained in
$I$.

\smallskip

We next give explicit conditions for $S$-primeness.  We begin with ideals.
The next result follows from Proposition~\ref{S-prime module sg} below,
by way of Lemma~\ref{S-prime ideal vs S-prime module}.  This result is
in fact true for any torsionfree group $G$; this follows from Corollary 4.6
in Passman \cite{Pas84} or Corollary 8.5 in Passman \cite{Pas89} after
passing to $R/I\sub S/IS$.

\begin{prop}\label{S-prime ideal sg}
Let $G$ be a \upg{} and suppose $S$ is a strongly $G$-graded ring with
$R=S_e$.  Let $I$ be a \GGi{} ideal of $R$. Then the following are
equivalent.
\begin{enumerate}
\item
$I$ is $S$-prime.
\item
If $a,b\in R$ and $S_gaS_{(hg)^{-1}}bS_h\sub I$ for all $g,h\in G$,
then $a\in I$ or $b\in I$.
\item
If $A,B$ are \GGi{} ideals of $R$ and $AB\sub I$, then $A\sub I$ or
$B\sub I$.
\qedsymbol
\end{enumerate}
\end{prop}

\medskip

We now turn to $S$-modules.

\begin{lem}\label{ann of induced module sg}
Let $G$ be a \upg{} and suppose $S$ is a strongly $G$-graded ring with
$R=S_e$.  Let $M$ be an $R$-module with annihilator $J$.
\begin{enumerate}
\item
$\ann_S (M\otimes_R S)=\st{J}{G}S$.
\item
If $M$ is $S$-prime, then $\st{J}{G}$ is $S$-prime.
\end{enumerate}
\end{lem}

\begin{proof}
(1) Let $I=\st{J}{G}$.  Then $(M\otimes S_g)I\sub M\otimes
S_gS_{g^{-1}}JS_g=MJ\otimes S_g=0$, so $(M\otimes_R S)IS=0$. This
shows $IS\sub\ann (M\otimes_R S)$.

For the reverse inclusion, suppose $(M\otimes_R S)s=0$ for some $s\in
S_g$. Then $M\otimes S_{h^{-1}}sS_{g^{-1}}S_{h}=0$, whence
$S_{h^{-1}}sS_{g^{-1}}S_{h}\sub J$.  Thus $sS_{g^{-1}}\sub
S_hJS_{h^{-1}}$ for all $h\in G$, and so $s\in IS_g$. This shows $\ann
(M\otimes_R S)\sub IS$.

(2) This follows immediately from (1).
\end{proof}

\begin{prop}\label{S-prime module sg}
Let $G$ be a \upg{} and suppose $S$ is a strongly $G$-graded ring with
$R=S_e$.  Let $M$ be a nonzero $R$-module.  Then the following
conditions are equivalent.
\begin{enumerate}
\item
$M$ is $S$-prime.
\item
If $m\in M$, $b\in R$, and $m(S_gbS_{g^{-1}})=0$ for all $g\in G$,
then $m=0$ or $Mb=0$.
\item
If $m\in M$, $B$ is a \GGi{} ideal, and $mB=0$, then $m=0$ or
$MB=0$.
\item
For every nonzero submodule $N$ of $M$, we have $\st{\ann
N}{G}=\st{\ann M}{G}$.
\end{enumerate}
\end{prop}

\begin{proof}
$(1)\implies (4)$ Suppose $M$ is $S$-prime and $N$ is a nonzero
submodule of $M$.  Then $\st{\ann N}{G}S=\ann (N\otimes_R S)=\ann
(M\otimes_R S)=\st{\ann M}{G}S$, so  $\st{\ann N}{G}=\st{\ann
M}{G}$.

$(4)\implies (3)$ Suppose (4) holds, let $m\in M$ be nonzero, and
suppose $mB=0$ for a \GGi{} ideal $B$.  Then $B\sub \st{\ann mR}{G}=
\st{\ann M}{G}$. This implies $MB=0$.

$(3)\implies (2)$  Suppose $m\in M$, $b\in R$, and
$m(S_gbS_{g^{-1}})=0$ for all $g\in G$.  Set $B=\sum_{g\in G}
S_gbS_{g^{-1}}$.  Then $B$ is a \GGi{} ideal of $R$ and $mB=0$. Thus
either $m=0$ or $MB=0$, and the latter equality implies $Mb=0$.

$(2)\implies (1)$  Suppose (2) holds and let $x\in M\otimes_R S$ be
(nonzero) homogeneous of degree $g$.  By Lemma~\ref{graded prime},
we need to show that if $xSs=0$ for a homogeneous $s\in S$, then
$(M\otimes_R S)s=0$. Let $h$ be the degree of $s$.  Then
$0=xS_isS_{(gih)^{-1}}=xS_{g^{-1}}S_{gi}sS_{h^{-1}}S_{(gi)^{-1}}$
for all $i\in G$. If $xS_{g^{-1}}=0$, then $x\in xR=xS_{g^{-1}}S_g=0$.
This is impossible, so there is a nonzero $m\in M$ with $m\otimes 1\in
xS_{g^{-1}}$. We can make $gi$ arbitrary, whence
$m(S_{j}sS_{h^{-1}}S_{j^{-1}})=0$ for all $j\in G$. It follows from (2)
that $MsS_{h^{-1}}=0$.  Since elements of $S_gsS_g^{-1}$ are also
homogeneous, the same argument implies $MS_gsS_{(hg)^{-1}}=0$ for
all $g\in G$. Thus $M\otimes S_gs=0$ for all $g$, whence $(M\otimes_R
S)s=0$.
\end{proof}

\medskip

We can use this result to describe the set of associated primes of induced
modules.

\begin{prop}\label{ass MS sg}
Let $G$ be a \upg{} and suppose $S$ is a strongly $G$-graded ring with
$R=S_e$.  Let $M$ be an $R$-module.  Then
\[
\Ass (M\otimes_R S) =
\{\,\st{J}{G}S\mid J\text{ is the annihilator of an $S$-prime submodule of
}M\,\}.
\]
If $R$ \hasacc, then $\Ass (M\otimes_R S) =
\{\,\st{J}{G}S\mid J\in \Ass M\,\}$.
\end{prop}

\begin{proof}
Let $M$ be an $R$-module and $N$ be a prime submodule of
$M\otimes_R S$.  Since $S$ is strongly graded, Corollary~\ref{submods
of homog mods} implies there is a nonzero $a\in M$ such that $(a\otimes
1)S=aR\otimes_R S$ is isomorphic to a submodule of $N$. Set $\ann_R
aR=J$; since $N$ is prime, Lemma~\ref{ann of induced module sg} tells
us $\ann N=\ann (a\otimes 1)S= \st{J}{G}S$.

If $R$ \hasacc, the module $aR$ contains a nonzero submodule $K$ with
maximal annihilator.  In particular, $K$ is prime and $\ann K\in\Ass M$.
By Proposition~\ref{S-prime module sg}, $\st{J}{G}=\st{\ann K}{G}$.
\end{proof}

Since all graded modules are induced, we could re-write
Proposition~\ref{ass MS sg} as follows. If $N$ is a graded $S$-module,
then $\Ass N =
\{\,\st{J}{G}S\mid J\text{ is the annihilator of an $S$-prime submodule of
}N_e\,\}$.

\section{The crossed product case}
\label{The crossed product case}

In this section, we assume $S=\skewm{R}{G}$ is a crossed product. We
will give conditions for primeness of $IS$ and $M\otimes_R S$, and we
will show that prime submodules of $M\otimes_R S$ contain submodules
of the form $L\otimes_R S$, but in the case $S=\skewm{R}{G}$, the
subset $L$ is not necessarily a submodule of $M$. Instead, it might be a
twist of a submodule over a subring $\ssa{g}{R}$ of $R$.  If \psurj, we can
take $L$ to be an $S$-prime submodule of $M$ in the presence of an
extra reversibility hypothesis.

The crossed product case is the most complicated case considered in this
paper, and there are many subsidiary hypotheses one can make to get
better results.  Among these assumptions are commutativity of $R$,
surjectivity of the maps $\sigma(g)$, the a.c.c.\ on ideals of $R$, and a
reversibility condition on $G$ and $\sigma$ defined below (which holds,
e.g., if $G$ or the image $\sigma(G)$ is commutative).   In each of the
three parts of the chapter (primeness of ideals, primeness of modules,
associated primes), we prove a general statement with no extra
hypotheses and then we prove statements in the presence of some
combination of extra hypotheses.  In Section~\ref{Examples}, we give
examples showing that (some) of our results fail without these extra
hypotheses.

For example, Proposition~\ref{S-prime module for R*G} gives necessary
and sufficient conditions for the induced module
$\skewm{M}{G}=M\otimes_R S$ to be a prime $S$-module. One
condition is an elementwise condition and another is an idealwise condition.
In Corollary~\ref{S-prime module for R*G second}, we show that
primeness implies other conditions that are then shown in
Corollary~\ref{S-prime module R*G comm or surj} to be equivalent in the
presence of reversibility when $R$ is commutative or \psurj{}.  In
Corollary~\ref{S-prime module R*G acc} essentially these same
conditions are shown to be equivalent when $R$ \hasacc{} and \psurj{}.

\medskip

A crossed product $S=\skewm{R}{G}$ is a ring that is determined as
follows by a $4$-tuple $(R,G,\sigma,c)$, where $R$ is a ring, $G$ is a
monoid with identity $e$, and $\sigma:G\to\End_{\text{Ring}}(R)$ and
$c:G\times G\to GL_1(R)$ are functions.  We denote $\sigma(g)(r)$ by
$\ssa{g}{r}$.
\begin{enumerate}
\item
The ring $S=\oplus_{g\in G} R\basis{g}$ is a free left $R$-module with
basis $\{\,\basis{g}\smid g\in G\,\}$;
\item
Multiplication is determined by the relations
$\basis{g}r=\ssa{g}{r}\basis{g}$ and
$\basis{\vphantom{h}g}\mspace{1mu}\basis{h}=c(g,h)\basis{gh}$;
\item
The following conditions are satisfied for all $r\in R$, $g,h,i\in G$:
\begin{enumerate}[(a)]
\item
$\ssa{e}{r}=r$;
\item
$c(g,e)=c(e,g)=1$;
\item
$\ssaa{g}{h}{r}c(g,h)=c(g,h)\ssa{gh}{r}$.
\item
$c(g,h)c(gh,i)=\ssa{g}{c(h,i)}c(g,hi)$.
\end{enumerate}
\end{enumerate}
Conditions (c) and (d) guarantee that the multiplication in $S$ is
associative and conditions (a) and (b) guarantee that $1\basis{e}$ is the
multiplicative identity of $S$. The condition that the values $c(g,h)$ are
units is not required for $S$ to be a ring with identity, but this condition is
customary and is necessary for our later results to hold.  See Passman
\cite[pp.~1--3]{Pas89} for more details on crossed products, but note that
Passman writes coefficients on the right of basis elements, that is, he has
$S=\oplus_{g\in G} \basis{g}R$.  Thus his formulas are generally the
``reverse'' of ours.

If $M$ is a right $R$-module, the induced $S$-module $M\otimes_R S$
can be identified with $\oplus_{g\in G} M\basis{g}$ with multiplication
$m\basis{g}\cdot r\basis{h}=m\ssa{g}{r}c(g,h)\basis{gh}$.  We will
denote this $S$-module by $\skewm{M}{G}$.

\smallskip

Condition (c) above can be stated as
$\sigma(g)\cp\sigma(h)=\tau_{g,h}\cp\sigma(gh)$ for all $g,h\in G$,
where $\tau_{g,h}$ is the inner automorphism defined as conjugation by
$c(g,h)$.

For any $X\sub R$, let us define $\ssa{g}{X}=\{\,\ssa{g}{x}\smid x\in
X\,\}$.   We say \define{\psurj} if $\sigma(g)$ is surjective for all $g\in
G$. Note that \psurj{} if and only if $I^g\normal R$ whenever $I\normal
R$. Since an ideal is invariant under inner automorphisms (and inner
automorphisms are trivial when $R$ is commutative), we get the simpler
statement $\ssaa{g}{h}{I}=\ssa{gh}{I}$ for ideals $I$ when either
\psurj{} or $R$ is commutative.

\smallskip

For $g,h\in G$, define $g\equiv h$ if there is an inner automorphism
$\tau$ such that $\sigma(g)=\tau\cp\sigma(h)$.  The relation $\equiv$ is
a congruence, that is, it is an equivalence relation and whenever $g\equiv
i$, $h\equiv j$ hold, we have $gh\equiv ij$.  We say \define{\rev} if for
any $g,h\in G$, there exist $h_1,h_2\in G$ such that $gh\equiv h_1g$
and $hg\equiv gh_2$.  We refer to this condition as ``\revsh'' for short.
\Revsh{} is equivalent to: for all $g,h\in G$, there exist $h_1,h_2\in
G$ such that
$\sigma(g)\cp\sigma(h)=\tau_1\cp\sigma(h_1)\cp\sigma(g)$ and
$\sigma(h)\cp\sigma(g)=\tau_2\cp\sigma(g)\cp\sigma(h_2)$ for some
inner automorphisms $\tau_1,\tau_2$ of $R$.  A way to symbolically
state that \rev{} is ``$\Inn R\cp\sigma(gG)=\Inn R\cp\sigma(Gg)$'' for
all $g\in G$.

If either $R$ is commutative or \psurj, we can eliminate the conjugation
when operating on ideals and thus if \rev{}, then for all $g,h\in G$ and all
$I\normal R$, there exist $h_1,h_2\in G$ such that
$\ssa{gh}{I}=\ssa{h_1g}{I}$ and $\ssa{hg}{I}=\ssa{gh_2}{I}$. If
$\sigma(h)$ is always the identity, e.g., in the monoid ring case, then \rev.
\Revsh{} is also satisfied if $gG=Gg$ for all $g$, which occurs if $G$ is a
group or $G$ is commutative. Another example of a monoid $G$ with
$gG=Gg$ for all $g\in G$ is the monoid generated by $x,y,z,z^{-1}$
subject to the relations that $z$ is central and $yx=xyz$. This monoid
embeds in a torsionfree nilpotent group, so it is a \upm. However, if we
leave out $z^{-1}$ in the definition of $G$, we get a \upm{} that does not
satisfy $Gg=gG$.

\medskip

The terminology used to describe whether a map takes an ideal into itself
or does something stronger is rather variable in the literature.  We will
adopt the following terminology.  We say an ideal $I$ of $R$ is
\define{\GGs} if $\ssa{g}{I}\sub I$ for all $g\in G$, and we say $I$ is
\define{\GGi} if $\sigma(g)^{-1}(I)=I$ for all $g\in G$.

\begin{lem}\label{s-good conditions for R*G}
Let $I$ be an ideal of $R$ and $S=\skewm{R}{G}$.  Then $I$ is right
$S$-stable if and only if $I$ is \GGs.
\end{lem}

\begin{proof}
If we look at terms of degree $g$, we see $SI\sub IS$ if and only if
$\ssa{g}{I}\basis{g}=\basis{g}I\sub I\basis{g}$ for all $g\in G$.  This
last containment holds if and only if $\ssa{g}{I}\sub I$ for all $g\in G$.
\end{proof}

If $X\sub R$, we define
\[
\st{X}{G}=\cap_{g\in G}\,\sigma(g)^{-1}(X) =
\{\,r\in R\mid \ssa{g}{r}\in X\text{ for all }g\in G\,\}.
\]
If $X$ is an ideal of $R$, so is $\st{X}{G}$. It is not hard to see that
$\st{X}{G}$ is the largest \GGs{} subset of $R$ contained in $X$.

\smallskip

We now state two results that tell us certain \GGs{} ideals are necessarily
\GGi.  The first result is a standard application of the ascending chain condition.

\begin{lem}\label{s-stable implies s-good with acc}
Suppose $S=\skewm{R}{G}$, \psurj, and $R$ \hasacc. If $I$ is a \GGs{}
ideal of $R$, then $I$ is \GGi{} and $\ssa{g}{I}=I$ for all $g\in G$.
\qedsymbol
\end{lem}

If $I$ is $S$-prime, Lemma~\ref{s-good conditions for R*G} implies $I$
is \GGs.  The next lemma shows that often an $S$-prime ideal is
necessarily \GGi.  Examples~\ref{reversibility necessary for invariance
simpler} and~\ref{reversibility necessary for invariance auto} show that
some condition along the lines of \revsh{} must be imposed to make this
true.

\begin{lem}\label{S-prime implies s-good for R*G}
Suppose $S=\skewm{R}{G}$ and suppose that \rev.
\begin{enumerate}
\item
If $I$ is a right $S$-prime ideal of $R$, then $I$ is \GGi.
\item
If $M$ is an $S$-prime $R$-module, then $\st{\ann M}{G}$ is \GGi.
\end{enumerate}
\end{lem}

\begin{proof}
(2) Let $I=\st{\ann M}{G}$ and suppose $\ssa{g}{r}\in I$ for some $g\in
G, r\in R$.  Set $A=\sum_{h,i\in G} R\basis{hg\imath}$.  It is easy to see
that $A\normal S$, and $A\notsub IS$ since $1\notin I$.

By the reversibility condition, there exists $i'\in G$ such that $gi\equiv
i'g$. Thus there exists a unit $u$ such that
$\basis{hg\imath}r=\ssa{hg\imath}{r}\basis{hg\imath}=
u\ssaa{h\imath'}{g}{r}u^{-1}\basis{hg\imath}$. The coefficient
$u\ssaa{h\imath'}{g}{r}u^{-1}\in I$ because $I$ is
\GGs. It follows that $Ar\sub IS$.  Since $IS$ is prime, this implies $r\in
I$.

(1) This follows from (2).
\end{proof}

The next lemma that shows another way in which \revsh{} can be useful.

\begin{lem}\label{preservation of stability}
Suppose $S=\skewm{R}{G}$, $I$ is a \GGs{} ideal of $R$, $g\in G$, and
\rev.  Then the ideals $RI^gR$ and $\sigma(g)^{-1}(I)$ are \GGs.
\end{lem}

\begin{proof}
Set $A=RI^gR$ and $B=\sigma(g)^{-1}(I)$.  Suppose $h\in G$ and let
$h'\in G$ satisfy $hg\equiv gh'$.  Then for some unit $u\in R$, we have
$A^h=R^h(I^g)^hR^h=R^hu(I^{h'})^gu^{-1}R^h\sub RI^gR=A$,
where the containment holds because $I$ is \GGs.

Suppose that $b\in B$, so $b^g\in I$, and $h\in G$.   Let $h'\in G$ satisfy
$gh\equiv h'g$. Then for some unit $u\in R$, we have
$(b^h)^g=u(b^g)^{h'}u^{-1}\in uI^{h'}u^{-1}\sub I$, where the
containment holds because $I$ is \GGs.  Thus $b^h\in B$.
\end{proof}

\medskip

We now give explicit conditions for $S$-primeness.

\begin{prop}\label{S-prime ideal R*G}
Let $G$ be a \upm, $I$ an ideal of $R$, and $S=\skewm{R}{G}$.  If $I$
is \GGs, then the following conditions are equivalent.
\begin{enumerate}
\item
$I$ is right $S$-prime.
\item
If $a,b\in R$, $g\in G$, and $a\ssa{g}{R}\ssaa{g}{h}{b}\sub I$ for all
$h\in G$, then $a\in I$ or $\ssa{h}{b}\in I$ for all $h\in G$.
\item
If $A$ is a left ideal of $R$, $B$ is a \GGs{} ideal of $R$, and
$A\ssa{g}{B}\sub I$ for some $g\in G$, then $A\sub I$ or $B\sub I$.
\end{enumerate}
If $R$ is commutative or \psurj, we can replace ``left ideal'' by ``ideal'' in
(3).
\end{prop}

\begin{proof}
This follows from Proposition~\ref{S-prime module for R*G} below, by
way of Lemma~\ref{S-prime ideal vs S-prime module}.
\end{proof}

The first corollary below gives a more appealing condition for an ideal to be
$S$-prime when \psurj{} or $R$ is commutative and \revsh{} holds.
Example~\ref{prime doesn't imply S-prime} shows that the corollary does
not hold for an arbitrary crossed product, even when $I$ is \GGi{} and
\revsh{} holds. Example~\ref{M prime but MS no ass primes example}
shows that the corollary fails without the assumption that $I$ is \GGi.
Example~\ref{reversibility necessary for S-prime comm} shows that the
corollary fails in the commutative case without \revsh{}.  This corollary is
unique among our results in that we do not have a module analog when
\psurj.

\begin{cor}\label{S-prime ideal R*G comm or surj}
Let $G$ be a \upm, $I$ an ideal of $R$, and $S=\skewm{R}{G}$. Suppose
$I$ is \GGi{} and suppose that either (a) $R$ is commutative and
\rev{} \emph{or} (b) \psurj.  Then $I$ is
$S$-prime if and only if $AB\sub I$ implies $A\sub I$ or $B\sub I$
whenever $A,B$ are ideals of $R$ with $B$ \GGs{}.
\end{cor}

\begin{proof}
Under the assumption of \revsh, this result follows from
Corollary~\ref{S-prime module R*G comm or surj} below, by way of
Lemma~\ref{S-prime ideal vs S-prime module}.

Let us drop \revsh{} assumption and suppose that \psurj{} and that
whenever $A,B$ are ideals of $R$ with $B$
\GGs{} and $AB\sub I$, we have $A\sub I$ or $B\sub I$. Suppose that
$A$ is an ideal of $R$, $B$ is a \GGs{} ideal of $R$, and $A\ssa{g}{B}\sub
I$ for some $g\in G$.  Since $\sigma(g)$ is surjective, if we set
$\tilde{A}=\sigma(g)^{-1}(A)$, then $\ssa{g}{\tilde{A}}=A$. Certainly,
$\ssa{g}{(\tilde{A}B)}\sub I$; since $I$ is \GGi{}, this implies
$\tilde{A}B\sub I$.  By our hypothesis, this implies $\tilde{A}\sub I$ or
$B\sub I$.  If the former containment holds, then
$A=\ssa{g}{\tilde{A}}\sub \ssa{g}{I}\sub I$.  Thus by
Proposition~\ref{S-prime ideal R*G}, $I$ is $S$-prime.
\end{proof}

The next corollary give situations where \revsh{} hypothesis is not
needed.

\begin{cor}\label{S-prime ideal R*G acc}
Let $G$ be a \upm, $I$ a \GGi{} ideal of $R$, and $S=\skewm{R}{G}$.
Suppose $R$ \hasacc{} and \psurj. Then $I$ is $S$-prime if and only if
whenever $J,K$ are \GGi{} ideals of $R$ with $JK\sub I$, we have
$J\sub I$ or $K\sub I$.
\end{cor}

\begin{proof}
This result follows from Corollary~\ref{S-prime module R*G acc} below,
by way of Lemma~\ref{S-prime ideal vs S-prime module}.
\end{proof}

\bigskip

We now turn to modules.

\begin{lem}\label{S-prime anns for R*G}
Let $G$ be a \upm{} and $S=\skewm{R}{G}$.  Let $M$ be an
$R$-module with annihilator $J$.
\begin{enumerate}
\item
$\ann_S \skewm{M}{G}=\skewm{\st{J}{G}}{G}$.
\item
If $M$ is $S$-prime, then $\st{J}{G}$ is right $S$-prime.
\qedsymbol
\end{enumerate}
\end{lem}

\begin{proof}
(1) Let $I=\st{J}{G}$.  Then $(M\basis{g})I\sub M\ssa{g}{I}\basis{g}
\sub MJ\basis{g}=0$, so $(\skewm{M}{G})\st{J}{G}=0$. This
shows $\st{J}{G}\sub\ann \skewm{M}{G}$.

For the reverse inclusion, suppose $\st{\ann M}{G}r\basis{h}=0$ for
some $r\in R,h\in G$. Then
$M\ssa{g}{r}c(g,h)\basis{gh}=(M\basis{g})r\basis{h}=0$, whence
$M\ssa{g}{r}=0$ and $\ssa{g}{r}\in J$ for all $g\in G$.  This shows
$\ann \skewm{M}{G}\sub \st{J}{G}$.

(2) This follows immediately from (1).
\end{proof}

\smallskip

We next give conditions for $S$-primeness of a module when
$S=\skewm{R}{G}$.  We begin with two results in the general case and
then proceed to results with extra hypotheses.

\begin{prop}\label{S-prime module for R*G}
Let $G$ be a \upm{} and $S=\skewm{R}{G}$.  Let $M$ be a nonzero
$R$-module. The following conditions are equivalent.
\begin{enumerate}
\item
$M$ is $S$-prime.
\item
If $m\in M$, $b\in R$, $g\in G$, and $m\ssa{g}{R}\ssaa{g}{h}{b}=0$
for all $h\in G$, then either $m=0$ or $M\ssa{h}{b}=0$ for all $h\in
G$.
\item
If $m\in M$, $B$ is a \GGs{} ideal of $R$, and $m\ssa{g}{B}=0$ for
some $g\in G$, then $m=0$ or $MB=0$.
\end{enumerate}
If $R$ is commutative or \psurj, we can add the following equivalent
condition:
\begin{enumerate}
\item[(4)]
If $N$ is a submodule of $M$,  $B$ is a \GGs{} ideal of $R$, and
$N\ssa{g}{B}=0$ for some $g\in G$, then $N=0$ or $MB=0$.
\end{enumerate}
\end{prop}

\begin{proof}
$(1)\implies (3)$  Suppose $M$ is $S$-prime and $m\ssa{g}{B}=0$ for
some nonzero $m\in M$, some $g\in G$, and some \GGs{} ideal $B$.
Then $BS$ is an ideal of $S$ and
$(m\basis{g}S)BS=m\basis{g}BS=m\ssa{g}{B}\basis{g}S=0$. Since
$\skewm{M}{G}$ is prime, $\ann m\basis{g}S=\ann\skewm{M}{G}$.
Thus $MB=0$.

$(3)\implies (2)$  Suppose $m\in M$, $b\in R$, $g\in G$, and
$m\ssa{g}{R}\ssaa{g}{h}{b}=0$ for all $h\in G$.  Set $B=\sum_{h\in
G}R\ssa{h}{b}R$.  Then $B$ is a \GGs{} ideal of $R$ and
$m\ssa{g}{B}=0$. Thus either $m=0$ or $MB=0$, and the latter equality
implies $M\ssa{h}{b}=0$ for all $h\in G$.

$(2)\implies (1)$  Suppose (2) holds and let $m\basis{g}\in
\skewm{M}{G}$ be a (nonzero) homogeneous element.
By Lemma~\ref{graded prime}, we need to show that if
$m\basis{g}Sb\basis{\imath}=0$ for $b\in R, i\in G$, then
$\skewm{M}{G} b\basis{\imath}=0$. Suppose
$m\basis{g}Sb\basis{\imath}=0$; this implies that for any $r\in R,h\in
G$ we have
\[
0=m\basis{g}r\basis{h}b\basis{\imath}
=m\ssa{g}{r}\basis{g}\ssa{h}{b}\basis{h}\basis{\imath}=
m\ssa{g}{r}\ssaa{g}{h}{b}\ssa{g}{c(h,i)}c(g,hi)\basis{gh\imath}.
\]
Thus $m\ssa{g}{R}\ssaa{g}{h}{b}=0$ for all $h\in G$ since $c$ always
yields units.  Thus by (2), we can conclude that $M\ssa{h}{b}=0$ for all
$h\in G$. This implies that $(\skewm{M}{G})b=0$, and this proves
$\skewm{M}{G}$ is prime.

$(3)\implies (4)$ This is clearly true without any extra hypotheses.

\smallskip

Assume $R$ is commutative or \psurj, and suppose $(4)$ holds.  We prove
$(3)$.  Suppose $m\in M$, $B$ is a \GGs{} ideal of $R$, and
$m\ssa{g}{B}=0$ for some $g\in G$.  If $R$ is commutative,
$mR\ssa{g}{B}=m\ssa{g}{B}R=0$, while if \psurj, $\ssa{g}{B}\normal
I$, so $mR\ssa{g}{B}=m\ssa{g}{B}=0$.  Thus if we set $N=mR$, we have
$N\ssa{g}{B}=0$, and so we may apply condition (4) to obtain $m=0$ or
$MB=0$.
\end{proof}

\begin{cor}\label{S-prime module for R*G second}
Let $G$ be a \upm{} and $S=\skewm{R}{G}$.  Let $M$ be a nonzero
$R$-module. Then $(1)\implies (2)\iff (3)\iff (4)$ for the following
conditions.
\begin{enumerate}
\item
$M$ is $S$-prime.
\item
If $m\in M$, $b\in R$, and $mR\ssa{h}{b}=0$ for all $h\in G$, then
either $m=0$ or $Mb=0$.
\item
If $m\in M$ and $B$ is a \GGs{} ideal of $R$ with $mB=0$, then
$m=0$ or $MB=0$.
\item
$\st{\ann N}{G}=\st{\ann M}{G}$ for every nonzero submodule $N$
of $M$.
\end{enumerate}
\end{cor}

\begin{proof}
$(1)\implies (3)$ This follows from $(1)\implies (3)$ in
Proposition~\ref{S-prime module for R*G} by setting $g=e$.

$(3)\implies (4)$ Let $N$ be a nonzero submodule of $M$, let $n\in N$ be
nonzero, and set $B=\st{\ann N}{G}$.  Then $nB=0$ and $B$ is \GGs,
and so by (2), $MB=0$. Thus $\st{\ann N}{G}\sub \st{\ann M}{G}\sub
\st{\ann N}{G}$; this proves (3).

$(4)\implies (2)$ Let $m\in M\smin\{0\}$ and $b\in R$ satisfy
$mR\ssa{h}{b}=0$ for all $h\in G$.  Set $B=\sum_{h\in G}
R\ssa{h}{b}R$, so that $B$ is a \GGs{} ideal of $R$.  Clearly $mB=0$.
Thus $B\sub \st{\ann mR}{G}=\st{\ann M}{G}$, whence $MB=0$ and
so $Mb=0$.

$(2)\implies (3)$ Let $m\in M$ and $B$ be a \GGs{} ideal of $R$ with
$mB=0$.  If $b\in B$, then $mRb^h=0$ for all $h\in G$.  By (2),
$Mb=0$.  This holds for any $b\in B$, whence $MB=0$.
\end{proof}

The next corollary states that the four conditions in
Corollary~\ref{S-prime module for R*G second} are equivalent when
several extra conditions hold.  Example~\ref{prime doesn't imply
S-prime} shows that the corollary does not hold for an arbitrary crossed
product, even when $\st{\ann M}{G}$ is
\GGi{} and \revsh{} holds. Example~\ref{M prime but MS no ass primes
example} shows that the hypothesis ``$\st{\ann M}{G}$ is \GGi'' is
necessary in parts~(2) and~(3) of the corollary.
Example~\ref{reversibility necessary for S-prime comm} shows that the
corollary fails without \revsh{}, at least in the commutative case.

\begin{cor}\label{S-prime module R*G comm or surj}
Let $G$ be a \upm{} and $S=\skewm{R}{G}$.  Let $M$ be a nonzero
$R$-module.  Assume that $\st{\ann M}{G}$ is \GGi{} and
\rev{}.  If either \psurj{} or $R$ is commutative, then the four
conditions in Corollary~\ref{S-prime module for R*G second} are
equivalent.
\end{cor}

\begin{proof}
By Corollary~\ref{S-prime module for R*G second}, it is enough to prove
$(3)\implies (1)$.  To do so, suppose $m\in M$ is nonzero, $B$ is a
\GGs{} ideal, and $m\ssa{g}{B}=0$. Set $B'=R\ssa{g}{B}R$.  By
Lemma~\ref{preservation of stability}, $B'$ is \GGs.  If  $\sigma(g)$ is
surjective, then $B'=\ssa{g}{B}$ and if $R$ is commutative,
$B'=\ssa{g}{B}R$, whence $mB'=0$ in either case.  Thus by (3),
$MB'=0$, and so $B'\sub\ann M$. As $B'$ is \GGs, this implies
$\ssa{g}{B}\sub B'\sub\st{\ann M}{G}$. Since $\st{\ann M}{G}$ is
\GGi, this in turn implies $B\sub\st{\ann M}{G}$, whence $MB=0$.
This proves $M$ is $S$-prime by Proposition~\ref{S-prime module for
R*G}.
\end{proof}

We get the same result when $R$ \hasacc{} and \psurj{}; reversibility is
not required, nor must we assume anything about $\st{\ann M}{G}$.

\begin{cor}\label{S-prime module R*G acc}
Let $G$ be a \upm{} and $S=\skewm{R}{G}$.  Let $M$ be a nonzero
$R$-module.  Suppose that $R$ \hasacc{} and \psurj.  Then the following
conditions are equivalent.
\begin{enumerate}
\item
$M$ is $S$-prime.
\item
If $m\in M$, $b\in R$, and $mR\ssa{h}{b}=0$ for all $h\in G$, then
either $m=0$ or $Mb=0$.
\item
If $m\in M$ and $B$ is a \GGi{} ideal of $R$ with $mB=0$, then
$m=0$ or $MB=0$.
\item
$\st{\ann N}{G}=\st{\ann M}{G}$ for every nonzero submodule $N$
of $M$.
\end{enumerate}
\end{cor}

\begin{proof}
That $(1)\implies (2)\iff (3)\iff (4)$ is the content of
Corollary~\ref{S-prime module for R*G second} (since \GGs{} and
\GGi{} are equivalent by Lemma~\ref{s-stable implies s-good with acc}).

$(3)\implies (1)$ Suppose $m\in M$ is nonzero, $B$ is a \GGs{} ideal, and
$m\ssa{g}{B}=0$.  By Lemma~\ref{s-stable implies s-good with acc}, we
know $\ssa{g}{B}=B$, and so by (3), we conclude $MB=0$.  Thus $M$ is
$S$-prime by Proposition~\ref{S-prime module for R*G}.
\end{proof}

\smallskip

It is of some interest to know whether prime implies $S$-prime.  The next
corollary gives us a result in this direction.  It will be used in several
examples in Section~\ref{Examples}.

\begin{cor}\label{S-prime module from prime module R*G}
Let $G$ be a \upm{} and $S=\skewm{R}{G}$.  Let $M$ be a prime
$R$-module with annihilator $I$.  Suppose that $R$ is commutative or
\psurj. Then the following conditions are equivalent.
\begin{enumerate}
\item
$M$ is $S$-prime.
\item
If $b\in R, g\in G$, and $\ssaa{g}{h}{b}\in I$ for all $h\in G$, then
$b\in I$.
\item
If $B$ is a \GGs{} ideal of $R$ and $B^g\sub I$ for some $g\in G$, then
$B\sub I$.
\end{enumerate}
\end{cor}

\begin{proof}
$(1)\implies (2)$ Suppose $\ssaa{g}{h}{b}\in I$ for all $h\in G$.  Then
$\ssa{g}{R}\ssaa{g}{h}{b}\sub I$ for all $h$, whence
$m\ssa{g}{R}\ssaa{g}{h}{b}=0$ for all $h$, for any $m\in M$.  We may
choose $m\ne 0$ and apply Proposition~\ref{S-prime module for R*G} to
obtain $Mb^h=0$ for all $h\in G$.  In particular, $b\in \ann M=I$.

$(2)\implies (3)$  Suppose $B$ is a \GGs{} ideal of $R$ and $B^g\sub I$
for some $g\in G$.  For any $b\in B$ and $h\in G$, we have
$\ssaa{g}{h}{b}\in B^g\sub I$.  By (2), this implies $b\in I$, so $B\sub
I$.

$(3)\implies (1)$  We will prove condition (4) of Proposition~\ref{S-prime
module for R*G} holds, whence $M$ is $S$-prime.  Thus we suppose $N$
is a nonzero submodule of $M$, $B$ is a \GGs{} ideal of $R$, $g\in G$,
and $NB^g=0$.  Since $M$ is prime, this implies $B^g\sub\ann N=\ann
M=I$.  Thus by (3), $B\sub I$, which in turn implies $MB=0$.
\end{proof}

\medskip

We might hope to prove the analog of Proposition~\ref{ass MS sg} for
$S=\skewm{R}{G}$, namely that the associated prime ideals of
$\skewm{M}{G}$ are of the form $\skewm{\st{J}{G}}{G}$ where $J$ is
the annihilator of an $S$-prime submodule of $M$. Unfortunately, this
statement is not true in general; see Example~\ref{prime doesn't imply
S-prime} or Example~\ref{affine example}.  The result becomes true
when we replace ``submodule of $M$'' with ``submodule of some twist
$M_{\sigma(g)}$''. The result is true as stated if \psurj{} and either
\revsh{} holds or $R$ \hasacc.

Recall that if $M$ is an $R$-module and $\phi:R\to R$ is a ring
endomorphism, we can define a new $R$-module $M_{\phi}$ as follows:
$M_{\phi}$ is the same additive group as $M$, but the $R$-action is
twisted according to the rule $m*r=m\phi(r)$. It is easy to see that there
is a bijection between $\phi(R)$-submodules of $M$ (regarded as
$\phi(R)$-modules) and $R$-submodules of $M_{\phi}$, given by
$L\mapsto L_{\phi}$.  Clearly we have $\ann L_{\phi}=\phi^{-1}(\ann
L)$.  When $\phi$ is not surjective, there may be $\phi(R)$-submodules
$L$ of $M$ that are not $R$-submodules, that is, that are not closed
under the action of $R$.  (See Example~\ref{prime doesn't imply
S-prime}, for instance.)

If $L$ is an $\ssa{g}{R}$-submodule of $M$, then $L\basis{g}$ is an
$R$-submodule of $\skewm{M}{G}$. The map $a\mapsto a\basis{g}$ is
an $R$-module isomorphism from $L_{\sigma(g)}$ to $L\basis{g}$.

\medskip

We are now ready to describe the associated prime ideals of
$\skewm{M}{G}$.

\begin{prop}\label{prime submods of homog mods for R*G}
Let $G$ be a \upm{} and $S=\skewm{R}{G}$.  Let $M$ be an
$R$-module and $N$ be a prime submodule of $\skewm{M}{G}$.  Then
there is an $i\in G$ and an $\ssa{i}{R}$-submodule $L$ of $M$ such that
$L_{\sigma(i)}$ is $S$-prime.  Moreover, if we set $J=\ann
L_{\sigma(i)}=\sigma(i)^{-1}(\ann_R L)$ and $I=\st{J}{G}$, then $I$ is
$S$-prime and $\ann N=IS$.
\end{prop}

\begin{proof}
Let $a\in N\smin\{0\}$ be as in Lemma~\ref{anns of submods of homog
mods}, that is, $a=r_1\basis{g_1}+\dots+r_k\basis{g_k}$ is a canonical
form for $a$ with $k$ minimal among nonzero elements of $N$. Thus
$\ann a$ is homogeneous and $\ann a=\ann r_1\basis{g_1}$, so $aS\iso
r_1\basis{g_1}S$. If we set $L=r_1\ssa{g_1}{R}$, then
$r_1\basis{g_1}S\iso\skewm{L_{\sigma(g_1)}}{G}$.

Set $J=\ann_R ( L_{\sigma(g_1)})$.  Then $\ann N=\ann aS=\ann
\skewm{L_{\sigma(g_1)}}{G}=\st{J}{G}S$.
\end{proof}

\smallskip

The next lemma shows that when \revsh{} holds and \psurj, we do not
need to consider the twists of $M$.

\begin{lem}\label{twisted S-prime implies S-prime for R*G}
Let $G$ be a \upm{} and $S=\skewm{R}{G}$.  Let $M$ be an
$R$-module and let $g\in G$.  Suppose that \rev{} and that \psurj. If $L$
is an $S$-prime submodule of $M_{\sigma(g)}$, then $L$ is an
$S$-prime $R$-submodule of $M$.
\end{lem}

\begin{proof}
We first note that since $\sigma(g)$ is surjective, $L$ is closed under
multiplication by elements of $R$ and so is an $R$-submodule of $M$.
Suppose that $B$ is a \GGs{} ideal of $R$, $m\in L$, $m\ne0$, and
$mB=0$.  Let $C=\sigma(g)^{-1}(B)$.  By Lemma~\ref{preservation of
stability}, $C$ is \GGs.  Moreover $m*C=mC^g=mB=0$.  Since
$L_{\sigma(g)}$ is $S$-prime, Corollary~\ref{S-prime module R*G
comm or surj}(2) implies $L*C=LB=0$.
\end{proof}

\smallskip

The next proposition sums up our results on associated primes in the
crossed product case.

\begin{prop}\label{ass MS poly for R*G}
Let $G$ be a \upm{} and $S=\skewm{R}{G}$.  Let $M$ be an
$R$-module.
\begin{enumerate}
\item
$\Ass \skewm{M}{G}=
\\
\{\,\skewm{\st{\sigma(g)^{-1}(J)}{G}}{G}\mid g\in G\text{ and
}J\text{ is the annihilator of an $S$-prime submodule of
}M_{\sigma(g)}\,\}$.
\item
If \psurj{} and \rev{}, then $\Ass\skewm{M}{G}=
\\
\{\,\skewm{\st{J}{G}}{G}\mid
J\text{ is the annihilator of an $S$-prime submodule of }M\,\}$.
\item
If $R$ \hasacc{} and \psurj, then $\Ass\skewm{M}{G}=
\{\,\skewm{\st{P}{G}}{G}\mid P\in\Ass M\,\}$.
\end{enumerate}
\end{prop}

\begin{proof}
(1) This follows from Proposition~\ref{prime submods of homog mods for
R*G}.

(2) This follows from Proposition~\ref{prime submods of homog mods for
R*G} and Lemma~\ref{twisted S-prime implies S-prime for R*G}.

(3) First suppose $P\in\Ass M$, so $P=\ann L$ for some prime
submodule $L$ of $M$.  The ideal $\st{P}{G}$ is \GGs{} and
$\st{P}{G}=\st{\ann K}{G}$ for all nonzero submodules $K$ of $L$. Thus
$L$ is $S$-prime by Corollary~\ref{S-prime module R*G acc}, and so
$\skewm{\st{P}{G}}{G}\in\Ass\skewm{M}{G}$.

Next, let $\skewm{\st{J}{G}}{G}\in\Ass \skewm{M}{G}$, where
$J=\ann N$ for an $S$-prime submodule $N$ of $M$.  Let $L$ be a
nonzero submodule of $N$ such that $P=\ann L$ is as large as possible.  It
is well-known that in this case, $L$ is a prime module, so $P\in\Ass M$.
By Corollary~\ref{S-prime module R*G acc}, $\st{J}{G}=\st{P}{G}$.
\end{proof}

\smallskip

Proposition~\ref{ass MS poly for R*G}(1) describes $\Ass
\skewm{M}{G}$ in general, but it requires us to check a great many modules.
Lemma~\ref{prime implies twisted prime for R*G} below shows that once
we find a twist that is prime, any twist of that is automatically prime as
well, so we don't need to check twists of twists.  First, we need another
lemma.

\begin{lem}\label{S-prime, twists, and inner autos}
Let $G$ be a \upm{}, $S=\skewm{R}{G}$, $\phi$ be an endomorphism
of $R$, and $\tau$ be an inner automorphism of $R$.  Then the following
conditions are equivalent.
\begin{enumerate}
\item
$M_{\phi}$ is $S$-prime.
\item
$M_{\tau\cp\phi}$ is $S$-prime.
\item
$M_{\phi\cp\tau}$ is $S$-prime.
\end{enumerate}
\end{lem}

\begin{proof}
Note that if $\psi$ is an endomorphism of $R$, then
$M_{\psi\cp\phi}=(M_{\psi})_{\phi}$.  Since $\tau^{-1}$ is also inner,
this observation implies that if we prove $(1)\implies (2)$ for arbitrary
$\tau$, then $(2)\implies (1)$.  Applying (2) with $\phi=\id$, we see that
any $M_{\tau}$ is prime whenever $M$ is $S$-prime.  Replacing $M$ by
$M_{\phi}$ and using $M_{\phi\cp\tau}=(M_{\phi})_{\tau}$, we see
that $(1)\implies (3)$.  We then get $(3)\implies (1)$ as we get
$(2)\implies (1)$.

Thus we only need to prove $(1)\implies (2)$. Let $*$ denote the action in
$M_{\phi}$, that is $m*r=m\phi(r)$ and let $\mathbin{**}$ denote the
action in $M_{\tau\cp\phi}$, that is $m\mathbin{**}r=m\tau(\phi(r))$.

Suppose that $M_{\phi}$ is $S$-prime, $m$ is a nonzero element of $M$,
$B$ is a \GGs{} ideal of $R$, and $g\in G$.  Let $u$ be a unit of $R$ such
that $\tau(r)=uru^{-1}$ and suppose that $m\mathbin{**}B^g=0$. Then
$0=m\tau(\phi(B^g))=mu\phi(B^g)u^{-1}$.  Thus
$0=mu\phi(B^g)=mu*(B^g)$.  Since $mu\ne0$ and $M_{\phi}$ is
$S$-prime, we conclude that $0=M*B=M\phi(B)=Mu\phi(B)$.  This
implies $0=Mu\phi(B)u^{-1}=M\mathbin{**}B$.  This proves
$M_{\tau\cp\phi}$ is $S$-prime.
\end{proof}

\begin{lem}\label{prime implies twisted prime for R*G}
Let $G$ be a \upm{} and $S=\skewm{R}{G}$.  Let $M$ be an
$R$-module.  If $M$ is $S$-prime, then $M_{\sigma(h)}$ is $S$-prime
for all $h\in G$.  More generally, if $M_{\sigma(g)}$ is $S$-prime, then
$M_{\sigma(hg)}$ is $S$-prime for all $h\in G$.
\end{lem}

\begin{proof}
Suppose $M$ is $S$-prime, $m\in M$ is nonzero, $g,h\in G$, and $B$ is a
\GGs{} ideal of $R$ with $m*\ssa{h}{B}=0$.  Then there is a unit $u\in R$ with
$0=m\ssaa{g}{h}{B}=mu\ssa{gh}{B}u^{-1}$.  Thus $(mu)B^{gh}=0$
and $mu\ne0$, so by primeness, $MB=0$.  Since $B$ is \GGs, we have
$M*B=MB^g\sub MB=0$.  This proves $M_{\sigma(g)}$ is $S$-prime.

Now suppose $M_{\sigma(g)}$ is $S$-prime and $h\in G$.  Then
$(M_{\sigma(g)})_{\sigma(h)}=M_{\sigma(g)\cp\sigma(h)}=
(M_{\tau})_{\sigma(hg)}$, where $\tau$ is conjugation by $c(h,g)$.
Since $(M_{\sigma(g)})_{\sigma(h)}$ is $S$-prime by the first part of
the proof, we can conclude from Lemma~\ref{S-prime, twists, and inner
autos} that $M_{\sigma(hg)}$ is $S$-prime.
\end{proof}

\section{The skew polynomial and skew laurent cases}
\label{The skew polynomial and skew laurent cases}

In this section, $\sigma$ is a ring endomorphism of $R$ and
$S=R[x;\sigma]$, or $\sigma$ is an automorphism and
$S=R[x^{\pm1};\sigma]$. Skew polynomial rings form a special case of
the crossed products considered in Section~\ref{The crossed product
case}, with $G=\bbN$, $\sigma(g)=\sigma^g$, and $c(g,h)$ identically
$1$.  Note also that \psurj{} if and only if $\sigma$ is surjective.  Skew
laurent rings form a special case of the strongly graded rings considered in
Section~\ref{The strongly graded case}.  In the skew laurent case
$S_1=Rx$, so $S_1IS_{-1}=RxIRx^{-1}=\{\,\sigma(i)\mid i\in I\,\}$ for
$I\sub R$. Thus our previous use of the notation $\sigma(I)$ does not
conflict with our new use of $\sigma$.  Likewise, $\sigma^{-1}(I)$ is the
same set under either meaning of the notation.

We will switch to the more common terminology of ``\sss{}'' and ``\ssi{}''
in place of ``\GGs{}'' and ``\GGi{}'' and the notation $\st{I}{\sigma}$ in
place of $\st{I}{G}$.  Thus we say a subset $X$ of $R$ is \define{\sss} if
$\sigma(X)\sub X$, and we say $X$ is \define{\ssi} if
$\sigma^{-1}(X)=X$.

When $S=R[x;\sigma]$, the induced module $M\otimes_R S$ can be
identified with $M[x;\sigma]=\oplus_{n=0}^\infty Mx^n$.  In case
$S=R[x^{\pm1};\sigma]$, we identify  $M\otimes_R S$ with
$M[x^{\pm1};\sigma]=\oplus_{n=-\infty}^\infty Mx^n$.  We write
$M[x;\sigma]x^k$ for the $S$-module $\oplus_{n=k}^\infty Mx^n$,
and we will write $I[x;\sigma]$ or $I[x^{\pm1};\sigma]$ in place of $IS$
for an ideal $I$ of $R$.

As one would expect, the results in this section correspond to the results in
Sections~\ref{The crossed product case} and~\ref{The strongly graded
case}, generally with some simplifications.

\smallskip

\begin{lem}\label{s-good conditions}
Let $I\normal R$ and let $\sigma$ be an endomorphism of $R$.
\begin{enumerate}
\item
If $S=R[x;\sigma]$, then $I$ is right $S$-stable if and only if $I$ is
\sss.
\item
If $S=R[x;\sigma]$ and $I$ is right $S$-prime, then $I$ is \ssi.
\item
If $\sigma$ is an automorphism and $S=R[x^{\pm1};\sigma]$, then
$I$ is [right] $S$-stable if and only if $I$ is \ssi.
\end{enumerate}
\end{lem}

\begin{proof}
(1) \& (2) These follow from Lemma~\ref{s-good conditions for R*G}.

 (3) This follows from Lemma~\ref{sg-good conditions}.
\end{proof}

If $X\sub R$, we define
\[
\st{X}{\sigma}=\cap_{n=0}^\infty \sigma^{-n}(X) =
\{\,r\in R\mid \sigma^n(r)\in X\text{ for all }n\in\bbN\,\}
\ \ \text{and}\ \
\st{X}{\sigma^{\pm}}=\cap_{m=-\infty}^\infty \sigma^{m}(X).
\]
If $X$ is an ideal of $R$, so is $\st{X}{\sigma}$, and if $\sigma$ is
surjective and $X$ is an ideal of $R$, so is $\st{X}{\sigma^{\pm}}$.
Clearly $\st{X}{\sigma}$ is the largest \sss{} subset of $R$ contained in
$X$, and if $\sigma$ is an automorphism, then $\st{X}{\sigma^{\pm}}$
is the largest
\ssi{} subset of $R$ contained in $X$.

\smallskip

We now give explicit conditions for $S$-primeness in the skew polynomial
and skew laurent cases.  As noted in the introduction, finding such
conditions for skew polynomial rings was one of our motivations. In the
case $S=R[x;\sigma]$, definitions of $\sigma$-prime ideal seem to have
been given first in Goldie-Michler \cite{GM74}, assuming $\sigma$ is an
automorphism and $R$ is right Noetherian, in Pearson-Stephenson
\cite{PS77}, assuming $\sigma$ is an automorphism, and in Irving
\cite{Irv79}, assuming $R$ is commutative. None of these definitions
guarantee that $I[x;\sigma]$ is prime in the general case.

The Goldie-Michler definition is the following: a \ssi{} ideal $I$ is
$\sigma$-prime if whenever $J,K$ are \ssi{} ideals and $JK\sub I$,
either $J\sub I$ or $K\sub I$.  (The definition of $H$-prime in
Montgomery-Schneider \cite{MS99} is virtually the same, except that
they use ``stable'' in place of ``invariant''.) This is certainly the most
aesthetically pleasing definition, and it guarantees that
$I[x^{\pm1};\sigma]$ is prime in the the skew laurent extension
$S=R[x^{\pm1};\sigma]$. Unfortunately, it does not guarantee that
$I[x;\sigma]$ is prime in $S=R[x;\sigma]$, even if $\sigma$ is an
automorphism.

The following lemma gives equivalent conditions generalizing the Irving
and Pearson-Stephenson conditions.  Since our conditions, as well as those
in the rest of this section, are shaped by the assumption that coefficients of
elements of $S=R[x;\sigma]$ are written on the left and $xr=\sigma(r)x$,
they may be left--right reversed from the conditions stated in parts of the
literature.

\begin{prop}\label{S-prime ideal poly}
Let $\sigma$ be an endomorphism of $R$ and $S=R[x;\sigma]$.  Let $I$
be a \sss{} ideal of $R$.  Then the following conditions are equivalent.
\begin{enumerate}
\item
$I$ is right $S$-prime.
\item
If $a,b\in R$, $p\in\bbN$, and $a\sigma^p(R)\sigma^q(b)\sub I$ for
all $q\ge p$, then $a\in I$ or $b\in I$.
\item
If $A$ is a left ideal of $R$, $B$ is a \sss{} ideal of $R$, and
$A\sigma^p(B)\sub I$ for some $p\in\bbN$, then $A\sub I$ or
$B\sub I$.
\end{enumerate}
If $I$ is \ssi, and either $R$ is commutative or $\sigma$ is surjective,
then all of the above conditions are equivalent to:
\\
\hspace*{1em}
(4) If $A,B$ are ideals of $R$ with $B$ \sss{} and $AB\sub I$, then
$A\sub I$ or $B\sub I$.
\end{prop}

\begin{proof}
This is a special case of Proposition~\ref{S-prime ideal R*G} and
Corollary~\ref{S-prime ideal R*G comm or surj}.
\end{proof}

\smallskip

The conditions for being $S$-prime when $S=R[x^{\pm1};\sigma]$ are a
special case of those in Proposition~\ref{S-prime ideal sg}; here we state
them in slightly different terms.

\begin{prop}\label{S-prime ideal laurent}
Let $\sigma$ be an automorphism of $R$ and $S=R[x^{\pm1};\sigma]$ .
Let $I$ be a \ssi{} ideal of $R$.  Then the following conditions are
equivalent.
\begin{enumerate}
\item
$I$ is [right] $S$-prime.
\item
If $a,b\in R$ and $\sigma^p(a)R\sigma^q(b)\sub I$ for all
$p,q\in\bbZ$, then $a\in I$ or $b\in I$.
\item
If $A,B$ are \ssi{} ideals of $R$ with $AB\sub I$, then $A\sub I$ or
$B\sub I$.
\qedsymbol
\end{enumerate}
\end{prop}

\bigskip

We next give conditions for $S$-primeness of a module when
$S=R[x;\sigma]$.  The following result is a special case of
Proposition~\ref{S-prime module for R*G} and Corollaries~\ref{S-prime
module for R*G second} and~\ref{S-prime module R*G comm or surj}.

Example~\ref{prime doesn't imply S-prime} shows that conditions~(4)
and~(5) in Proposition~\ref{S-prime module poly} below are not
equivalent to the others for an arbitrary $\sigma$, even when $\st{\ann
M}{\sigma}$ is \ssi.  Example~\ref{M prime but MS no ass primes
example} shows that the hypothesis ``$\st{\ann M}{\sigma}$ is \ssi'' is
necessary for the equivalence of (4) and (5) to the other conditions even if
we assume $R$ is commutative and $\sigma$ is an automorphism. (If $R$
is an integral domain and $\sigma:R\to R$ is a non-injective ring
homomorphism, then $R$ is not $S$-prime but it does satisfy conditions
(4) and (5).  This is another example that the hypothesis ``$\st{\ann
M}{\sigma}$ is \ssi'' is necessary for the equivalence of (4) and (5) to the
other conditions.)

\begin{prop}\label{S-prime module poly}
Let $S=R[x;\sigma]$ where $\sigma$ is an endomorphism of $R$ and let
$M$ be a nonzero $R$-module.  Consider the following conditions.
\\
The implications $(1)\iff (2)\iff (3)\implies (4)\iff (5)$ are always true.
\\
If $\st{\ann M}{\sigma}$ is \ssi{} and either $\sigma$ is surjective or
$R$ is commutative, then all five conditions are equivalent.
\begin{enumerate}
\item
$M$ is $S$-prime.
\item
If $m\in M$, $b\in R$, $p\in\bbN$, and
$m\sigma^p(R)\sigma^q(b)=0$ for all $q\ge p$, then $m=0$ or
$M\sigma^k(b)=0$ for all $k\in\bbN$.
\item
If $m\in M$, $B$ is a \sss{} ideal of $R$, and $m\sigma^p(B)=0$ for
some $p\in\bbN$, then $m=0$ or $MB=0$.
\item
If $m\in M$ and $B$ is a \sss{} ideal of $R$ with $mB=0$, then $m=0$
or $MB=0$.
\item
If $N$ is a nonzero submodule of $M$, then $\st{\ann
N}{\sigma}=\st{\ann M}{\sigma}$.\qedsymbol
\end{enumerate}
\end{prop}

The next result follows from Corollary~\ref{S-prime module R*G acc}.

\begin{cor}\label{S-prime module poly acc}
Let $S=R[x;\sigma]$ where $\sigma$ is an automorphism of $R$ and let
$M$ be a nonzero $R$-module.  If $R$ \hasacc, then the following
conditions are equivalent.
\begin{enumerate}
\item
$M$ is $S$-prime.
\item
If $m\in M$ and $B$ is a \ssi{} ideal with $mB=0$, then $m=0$ or
$MB=0$.
\item
If $N$ is a nonzero submodule of $M$, then $\st{\ann
N}{\sigma}=\st{\ann M}{\sigma}$.
\qedsymbol
\end{enumerate}
\end{cor}

\smallskip

The next corollary is the analog of Corollary~\ref{S-prime module from
prime module R*G} in the skew polynomial case; it relates primeness and
$S$-primeness. It will be used in some examples in
Section~\ref{Examples}.

\begin{cor}\label{S-prime module from prime module poly}
Let $M$ be a prime $R$-module with annihilator $I$ and let
$S=R[x;\sigma]$ where $\sigma$ is an endomorphism of $R$.  Suppose
that $k\in\bbN$ and either $R$ is commutative or $\sigma$ is surjective.
\begin{enumerate}
\item
$M$ is $S$-prime if and only if $\cap_{q\ge p}\,\sigma^{-q}(I)=
\st{I}{\sigma}$ for all $p\in\bbN$.
\item
$M_{\sigma^k}$ is $S$-prime if and only if $\cap_{q\ge
p}\,\sigma^{-q}(I)=\st{\sigma^{-k}(I)}{\sigma}$ for all $p\in\bbN$.
\end{enumerate}
\end{cor}

\begin{proof}
When $R$ is commutative or $\sigma$ is surjective, condition (2) in
Proposition~\ref{S-prime module poly} becomes: if $m\ne 0$ and
$mR\sigma^q(b)=0$ for all $q\ge p$, then $b\in I$. Since $M$ is prime,
$mR\sigma^q(b)=0$ if and only if $\sigma^q(b)\in I$. Thus $M$ is
$S$-prime if and only if $\cap_{q\ge p}\sigma^{-q}(I)\sub I$ for all $p$.
Likewise, $M_{\sigma^k}$ is $S$-prime if and only if $\cap_{q\ge
p}\sigma^{-q}(I)\sub\sigma^{-k}(I)$ for all $p$.
\end{proof}

\medskip

We now turn to the case $S=R[x^{\pm1};\sigma]$.  The following is a
special case of the corresponding result for strongly graded rings,
Proposition~\ref{S-prime module sg}.

\begin{prop}\label{S-prime module laurent}
Let $S=R[x^{\pm1};\sigma]$ where $\sigma$ is an automorphism of $R$
and let $M$ be an $R$-module.  Then the following conditions are
equivalent.
\begin{enumerate}
\item
$M$ is $S$-prime.
\item
If $m\in M$, $b\in R$, and $mR\sigma^p(b)=0$ for all $p\in\bbZ$,
then $m=0$ or $Mb=0$.
\item
If $m\in M$, $B$ is a \ssi{} ideal of $R$, and $mB=0$, then $m=0$ or
$MB=0$.
\item
For every nonzero submodule $N$ of $M$, we have $\st{\ann
N}{\sigma^{\pm}}=\st{\ann M}{\sigma^{\pm}}$.
\qedsymbol
\end{enumerate}
\end{prop}

\medskip

As in Section~\ref{The crossed product case}, for an $R$-module $N$,
$N_{\sigma^k}$ is  the module with $R$-action $n*r=n\sigma^k(r)$.  If
$L$ is a $\sigma^k(R)$-submodule of $M$ and we regard $Lx^k$ as a
subset of $M[x;\sigma]$, then it is an $R$-submodule and the map
$a\mapsto ax^k$ is an $R$-module isomorphism from $L_{\sigma^k}$
to $Lx^k$.

\smallskip

We are now ready to describe the associated prime ideals of
$M[x;\sigma]$.   (The skew laurent case is already covered in
Proposition~\ref{ass MS sg}.)  Proposition~\ref{prime submods of homog
mods for R*G} tells us that if $M$ is an $R$-module and $N$ is a prime
submodule of $M[x;\sigma]$, then there is a $k\in\bbN$ and a
$\sigma^k(R)$-submodule $L$ of $M$ such that $L_{\sigma^k}$ is
$S$-prime and such that if we set $J=\sigma^{-k}(\ann L)$ and
$I=\st{J}{\sigma}$, then $I$ is $S$-prime and $\ann N=I[x;\sigma]$.
This, together with Lemma~\ref{twisted S-prime implies S-prime for
R*G}, implies the next result, which is a special case of
Proposition~\ref{ass MS poly for R*G}. Part (2) of the proposition is
Theorem~1.2 in Nordstrom \cite{Nor05} and part (3) is Corollary~1.5 in
that paper.

\begin{prop}\label{ass MS poly}
Let $S=R[x;\sigma]$ where $\sigma$ is an endomorphism of $R$ and let
$M$ be an $R$-module.
\begin{enumerate}
\item
$\Ass M[x;\sigma]=
\\
\{\,\st{\sigma^{-k}(J)}{\sigma}[x;\sigma]\mid k\in\bbN\text{ and
}J\text{ is the annihilator of an $S$-prime submodule of
}M_{\sigma^k}\,\}$.
\item
If $\sigma$ is surjective, then $\Ass M[x;\sigma]=
\\
\{\,\st{J}{\sigma}[x;\sigma]\mid
J\text{ is the annihilator of an $S$-prime submodule of }M\,\}$.
\item
If $R$ \hasacc{} and $\sigma$ is an automorphism, then $\Ass
M[x;\sigma]=
\\
\{\,\st{P}{\sigma}[x;\sigma]\mid P\in\Ass M\,\}$.
\qedsymbol
\end{enumerate}
\end{prop}

\medskip

In the general case of Proposition~\ref{ass MS poly}, we may have to
check all submodules of each twist $M_{\sigma^k}$.  We can reduce
this work a little bit in some cases.

\begin{cor}\label{prime implies twisted prime poly}
Let $S=R[x;\sigma]$ where $\sigma$ is an endomorphism of $R$ and
let $M$ be an $R$-module such that $M_{\sigma^n}$ is $S$-prime for
some $n\in\bbN$.  Then there is a nonnegative integer $k$ such that
$M_{\sigma^n}$ is $S$-prime if and only if $n\ge k$.  In this case,
$\Ass M[x;\sigma]=
\{\,\st{\sigma^{-k}(\ann M)}{\sigma}[x;\sigma]\,\}$.
\end{cor}

\begin{proof}
This follows from Lemma~\ref {prime implies twisted prime for R*G}
and Proposition~\ref{ass MS poly}.
\end{proof}

\section{Examples}\label{Examples}

In this section we give several examples, organized into subsections. In the
first subsection, Example~\ref{affine example} discusses the case where
$R$ is the coordinate ring of an affine algebraic set and $M$ is the simple
module corresponding to a point. In the second subsection, we give several
examples where associated primes come from submodules of twists of
$M$.  Example~\ref{prime doesn't imply S-prime} is of a prime ring $R$
that is not $S$-prime when $S=R[x;\sigma]$; it also provides an example
showing that conditions (4) and (5) in Proposition~\ref{S-prime module
poly} and  conditions (2) and (3) in Corollary~\ref{S-prime module for
R*G} are not always equivalent to $S$-primeness. Examples~\ref{M
prime but MS no ass primes example} and~\ref{MS prime but M no ass
primes example} show that one of $M$, $M[x;\sigma]$ may have
associated prime ideals while the other does not.  In the third subsection,
we consider \revsh{}.  Examples~\ref{reversibility necessary for
invariance simpler} and~\ref{reversibility necessary for invariance auto}
show that without \revsh, an $\skewm{R}{G}$-prime ideal need not be
\GGi.  Example~\ref{reversibility necessary for S-prime comm} shows that the
weaker conditions for $S$-primeness in the commutative case generally
fail without \revsh.

\bigskip

\bigskip

\textbf{An example from algebraic geometry}

\medskip

Our first example describes the case where $R=\cO(X)$ is the coordinate
ring of an affine algebraic set, $\sigma$ corresponds to a regular map from
$X$ to itself, and $M$ is the simple $R$-module corresponding to a point
$a\in X$.  The induced module $M[x;\sigma]$ always has a unique
associated prime, and it is always the annihilator of some twist
$M_{\sigma^k}$.  The $S$-primeness or twisted $S$-primeness of $M$,
as well as $\Ass M[x;\sigma]$ are controlled by the sequence
$\{\phi^n(a)\}_{n\in\bbN}$. If the sequence repeats, the starting point
of the repeating segment is the $k$ above and the length of the repeating
part determines the annihilator.  If the sequence never repeats, the value
of $k$ and the annihilator are determined by the Zariski closures of the
tails of the sequence. We close with a discussion of the special case
$R=F[t]$, $\sigma(t)=t^p$.

\begin{ex}\label{affine example}
Let $F$ be an algebraically closed field, let $X$ be an affine algebraic set
over $F$, let $R=\cO(X)$ be the coordinate ring of $X$, and let
$\phi:X\to X$ be a regular map.  There is an endomorphism $\sigma$ of
the $F$-algebra $R$ corresponding to $\phi$.  The maximal ideals of $R$
have the form $\fkm_a$ for points $a\in X$, and
$\sigma^{-1}(\fkm_a)=\fkm_{\phi(a)}$.

Let $a\in X$ and set $M=R/\fkm_a$.  Define a sequence of points of $X$
by $a_0=a$ and $a_n=\phi^n(a)$.  The claims below follow from
Corollary~\ref{S-prime module from prime module poly}

If the sequence $a_0,a_1,\dots$ is finite, then it eventually becomes
periodic of period $l$, starting at some $a_k$, so it has the form
$a_0,\dots,a_k,a_{k+1},\dots,a_{k+l-1},a_k,\dots$.  Let $J$ be the
intersection of the ideals $\fkm_{a_k},\dots,\fkm_{a_{k+l-1}}$. Then
$M_{\sigma^k}$ is $S$-prime, and $\ann
M[x;\sigma]x^k=J[x;\sigma]$.  By Corollary~\ref{prime implies twisted
prime poly}, $\Ass M[x;\sigma]=\{J[x;\sigma]\}$.

If the sequence is infinite, let $Y_n$ be the Zariski-closure of
$\{\,a_n,a_{n+1},\dots\,\}$ for each $n\in\bbN$ and let
$Y=\cap_{n=0}^\infty Y_n$.  Let $J_n$ be the ideal determined by the
algebraic set $Y_n$, that is, $J_n=\cap_{q\ge n}\fkm_{a_q}$. Then
$\ann M[x;\sigma]x^n=J_n[x;\sigma]$ for all $n\in\bbN$.

The module $M$ is $S$-prime if and only if $J_n=J_0$ for all $n$.  This is
true if and only if the original sequence is contained in $Y$. Likewise,
$M_{\sigma^n}$ is $S$-prime if and only if $a_n,a_{n+1},\dots\in Y$.
Since $R$ is noetherian, there is an $n$ such that $Y=Y_n$.  Let $k$ be
the smallest such integer $n$.  If $k=0$, then $M$ is $S$-prime.  If
$k>0$, then $M$ is not $S$-prime, but $M_{\sigma^k}$ is $S$-prime.
By Corollary~\ref{prime implies twisted prime poly}, $\Ass
M[x;\sigma]=\{J_k[x;\sigma]\}$.

\bigskip

For a concrete example, let $p$ be a prime number, $X=\bbA^1$, and
$\phi(x)=x^p$ for $x\in X$.  Then $R$ is the polynomial ring $F[t]$ with
endomorphism $\sigma(t)=t^p$ and $S=R[x;\sigma]$.  Let $a\in F$ and
set $M=R/(t-a)$.

Four distinct cases occur, unless $p$ divides the characteristic of $F$,
when case (c) below cannot occur. We introduce the following notation for
use in cases (b) and (c). Let $r$ be a positive integer, $\omega$ a
primitive $r\Th$ root of unity, and $p$ a prime that does not divide $r$.
Let $\ell$ be the order of $p$ as an element of $GL_1(\bbZ_r)$.  We
define $f_{\omega,p}\in F[t]$ by $f_{\omega,p}=\prod_{i=0}^{\ell -1}
(t-\omega^{p^i})$.  If $p$ is a primitive root modulo $r$, then
$f_{\omega,p}$ is the $r\Th$ cyclotomic polynomial.

In each case, $\Ass M[x;\sigma]$ is a singleton, containing the annihilator
ideal we describe.

\begin{enumerate}[(a)]
\item
Suppose $a=0$.  The sequence determined by $a,\phi$ is $0,0,\dots$.
Thus $M$ is $S$-prime with $\ann M[x;\sigma]=tS$.
\item
Suppose $a$ is a primitive $r\Th$ root of unity where $r$ is not
divisible by $p$, and let $\ell$ be the order of $p$ as an element of
$GL_1(\bbZ_r)$.  The sequence determined by $a,\phi$ is
$a,a^p,a^{p^2},\dots,a^{p^{\ell-1}},a,\dots$.  Thus $M$ is $S$-prime
with $\ann M[x;\sigma]=f_{a,p}S$, where $f_{a,p}$ is defined above.
\item
Suppose $a$ is a primitive $p^kr\Th$ root of unity where $k\ge1$ and
$r$ is not divisible by $p$, and let $\ell$ be the order of $p$ as an
element of $GL_1(\bbZ_r)$.  Let $b=a^{p^k}$, so $b$ is a primitive
$r\Th$ root of unity.  The sequence determined by $a,\phi$ is
$a,a^p,a^{p^2},\dots,a^{p^{k-1}},
b,b^p,b^{p^2},\dots,b^{p^{\ell-1}},b,\dots$. Thus  $M$ is not
$S$-prime but $M_{\sigma^k}$ is $S$-prime with $\ann
M[x;\sigma]x^k=f_{b,p}S$, where $f_{b,p}$ is defined above.
\item
Suppose $a\ne0$ is not a root of unity.  The sequence determined by
$a,\phi$ is $a,a^p,a^{p^2},\dots$, which has no repetition.  Thus every
subsequence is infinite and hence Zariski-dense. In this case, $M$ is
$S$-prime with $\ann M[x;\sigma]=0$.
\end{enumerate}
\end{ex}

\bigskip

\textbf{Twisted submodules and the lack of a relationship between prime
and $S$-prime}

\medskip

Our first example shows that even a prime ideal need not be $S$-prime in
the skew polynomial case.  It includes an instance of an $S$-prime
submodule of $M_{\sigma}$ that is not an $R$-submodule. The example
also shows that an $R$-module $M$ can satisfy conditions (4) and (5) of
Proposition~\ref{S-prime module poly} or conditions (2) and (3) of
Corollary~\ref{S-prime module R*G comm or surj} without being
$S$-prime.

\begin{ex}\label{prime doesn't imply S-prime}
Let $F$ be a field and let $R$ be the $F$-algebra generated by variables
$s_0,s_1,\dots$, $t_0,t_1,\dots$ subject to the relations $s_is_j=0$ if
$i\le j$ and $s_it_j=0$ if $i<j$.  (We can also allow $t_it_j=t_jt_i$ for all
$i,j$ without affecting any of the results below.)

The ring $R$ is prime.  To see this, let $m_1,m_2$ be (nonzero)
monomials with $a$ the final letter of $m_1$ and $b$ the initial letter of
$m_2$, and suppose $m_1Rm_2=0$.  In particular $m_1t_0m_2=0$.
Since the relations of $R$ are of degree $2$, this implies that either
$at_0=0$ or $t_0b=0$.  But $t_0$ does not annihilate any generator on
either side, so this is impossible.

Now define $\sigma:R\to R$ by $\sigma(s_i)=s_{i+1}$ and
$\sigma(t_i)=t_{i+1}$ for all $i$.  This map preserves the relations, so it
yields a well-defined ring endomorphism.  It is easy to see that $\sigma$ is
injective.  We will show $R$ is not $S$-prime for $S=R[x;\sigma]$.  As a
matter of fact, $S$ is not even semiprime.

To see this, note that $s_0\sigma(R)=Fs_0$, so
$s_0\sigma(R)\sigma^q(s_0)=0$ for all $q\ge1$.  This shows $R$ is not
$S$-prime, and one easily sees that $(Ss_0xS)^2=0$.

\smallskip

Note that $L=Fs_0$ is not a right ideal, but $L$ is a
$\sigma(R)$-submodule of $R$. The $R$-module $L_{\sigma}$ is
$S$-prime, since $l*\sigma^p(R)=l\sigma^{p+1}(R)=L$ for all nonzero
$l\in L$ and all $p\ge 0$.  If $I$ is the ideal generated by all the $s$'s and
$t$'s, then $\ann L_{\sigma}=I$. Thus even though $R_R$ contains no
$S$-prime submodules, $S=R[x;\sigma]$ contains the prime submodule
$s_0xS$, and hence $S_S$ has $I[x;\sigma]$ as an associated ideal.

\smallskip

Since $R$ is prime and $\sigma$ is injective, and hence $\st{\ann
R_R}{\sigma}=0$ is \ssi, condition (4) of Proposition~\ref{S-prime ideal
poly} and conditions (4) and (5) of Proposition~\ref{S-prime module poly}
are certainly satisfied with $M=R_R$, as are conditions (2) and (3) of
Corollary~\ref{S-prime module R*G comm or surj}.  This shows that the
conditions in Propositions~\ref{S-prime ideal poly} and~\ref{S-prime
module poly} and Corollary~\ref{S-prime module R*G comm or surj} are
not equivalent in general.
\end{ex}

\bigskip

The next example, in the case $n=\infty$, gives an $R$-module $M$ that
is prime but has no $S$-prime submodules, with $S=R[x;\sigma]$,  where
$R$ commutative and $\sigma$ an automorphism.  When $n>0$ is finite,
the twist $M_{\sigma^n}$ is $S$-prime even though $M$ is not
$S$-prime. When $n=\infty$, no submodule of any twist of $M$ is
$S$-prime and $M[x;\sigma]$ has no associated primes. The case where
$n=\infty$ is Example 5.15 in Leroy-Matczuk \cite{LM04} and essentially
the same as Example~2.2 in Nordstrom \cite{Nor05}.

\begin{ex}\label{M prime but MS no ass primes example}
Let $F$ be a field and set $R_n=F[t_{-n},\dots,t_0,t_1,\dots]$ for
$n\in\bbN$ and $R_{\infty}=F[\dots,t_{-1},t_0,t_1,\dots]$. Define an
endomorphism $\sigma$ of $R_n$ (including $n=\infty$) by
$\sigma(t_i)=t_{i+1}$ for each $i$ and let $I$ be the ideal generated by
$t_0,t_1,\dots$.  Then $\sigma^{-j}(I)$ is the ideal generated by
$t_{-j},t_{-(j-1)},\dots$ if $j\le n$.  For $j\ge n$, we have
$\sigma^{-j}(I)=\sigma^{-n}(I)$.  Let $S=R_n[x;\sigma]$.

Set $M=R_n/I$ and note that $M$ is a prime $R_n$-module.  If $n$ is
finite, then $\cap_{q\ge p}\,\sigma^{-q}(I)=\sigma^{-n}(I)$ whenever
$p\ge n$. Thus $M$ is $S$-prime if $n=0$ by Corollary~\ref{S-prime
module from prime module poly}(1).  If $n>0$, then $M$ is not
$S$-prime, but $M_{\sigma^n}$ is $S$-prime, by
Corollary~\ref{S-prime module from prime module poly}(2).

If $n=\infty$, then $\sigma$ is an automorphism and the ideals
$\sigma^{-q}(I)$ strictly increase as $q$ increases. In particular,
$\cap_{q\ge p}\,\sigma^{-q}(I)=\sigma^{-p}(I)$ for all $p$, and there is
no $k$ with $\sigma^{-p}(I)\sub \sigma^{-k}(I)$ for all $p$.  Thus by
Corollary~\ref{S-prime module from prime module poly}, no
$M_{\sigma^k}$ can be $S$-prime. Since $M$ is prime and the
conditions of Corollary~\ref{S-prime module from prime module poly}
depend only on the annihilator of $M$, the same considerations apply to all
nonzero submodules.  Thus by Proposition~\ref{ass MS poly},
$M[x;\sigma]$ has no associated primes.

Since $M$ is prime, $I=\st{\ann M}{\sigma}=\st{\ann N}{\sigma}$ for
all nonzero submodules $N$ of $M$.  Also, if $mJ=0$ for a nonzero $m\in
M$ and an ideal $J$ of $R$, we must have $MJ=0$. However, if $n>0$,
the ideal $I$ is not \ssi{} and $M$ is not $S$-prime.  This show that the
hypotheses ``$\st{\ann M}{G}$ is \GGi'' and ``$\st{\ann M}{\sigma}$ is
\ssi'' are necessary in Corollary~\ref{S-prime module R*G comm or surj},
parts~(2) and~(3) and in Proposition~\ref{S-prime module poly}.
\end{ex}

The next example gives an $R$-module $M$ that is $S$-prime but
contains no prime submodules, with $S=R[x;\sigma]$, $R$ commutative
and $\sigma$ an automorphism.   This is essentially Example~2.3 in
Nordstrom \cite{Nor05}; we include it for completeness.

\begin{ex}\label{MS prime but M no ass primes example}
Let $F$ be a field and set
$R=F[\dots,t_{-1},t_0,t_1,\dots]/(\dots,t_{-1}^2,t_0^2,t_1^2,\dots)$.
We will write $t_i$ and not $\overline{t_i}$ for elements of $R$. Define
an automorphism $\sigma$ of $R$ by $\sigma(t_i)=t_{i+1}$ for each $i$
and let $S=R[x;\sigma]$.

First we note that the module $R_R$ contains no prime submodules and
so $\Ass R_R=\emptyset$. To see this, suppose $f\in R$ and suppose the
variable $t_n$ does not occur in $f$.  Then $t_n\notin\ann f$, but
$(ft_n)t_n=0$. Thus $\ann f\subsetneq \ann ft_n$.

On the other hand, $R_R$ is $S$-prime, i.e., $S$ is prime and $\Ass
S_S=\{0\}$. To see this, first note that if $f,g\in R$ are nonzero and have
no variables in common, then $fg\ne0$.  Given any nonzero $f,g$, there is
a $p$ such that $f,\sigma^q(g)$ have no variables in common for all $q\ge
p$.  Thus $f\sigma^q(g)\ne0$ for all $q\ge p$.  This proves $R_R$ is
$S$-prime by Proposition~\ref{S-prime module poly}.
\end{ex}

\bigskip

\textbf{\Revsh}

\medskip

The next examples concern \revsh{} that we introduced prior to
Lemma~\ref{s-good conditions for R*G}. The first two examples show
that Lemma~\ref{S-prime implies s-good for R*G} need not hold when
$G$ is the free monoid on two generators and $S=\skewm{R}{G}$, even if
the base ring $R$ is commutative. The first gives an example where $R$ is
noetherian; the second gives an example where each $\sigma(g)$ is an
automorphism.  Example~\ref{reversibility necessary for S-prime comm}
gives an example where $I$ and $M=S/I$ are not $S$-prime despite the
fact that $R$ is commutative and the conditions of
Corollaries~\ref{S-prime ideal R*G comm or surj} and~\ref{S-prime
module R*G comm or surj} hold, except for \revsh.

\smallskip

\begin{ex}\label{reversibility necessary for invariance simpler}
Let $G$ be the free monoid on $v,w$, let $F$ be a field of characteristic
$0$, and let $R=F[t]$.  Define $\sigma:G\to\End_{F-\text{Alg}}(R)$ by
$\ssa{v}{b(t)}=b(t+1)$ and $\ssa{w}{b(t)}=b(0)$ for $b(t)\in F[t]$.
Since every $\sigma(g)$ is the identity on constants, it is easy to see that
for all $j\in\bbN$ and $g\in G$, we have $\ssaa{g}{wv^j}{b}=b(j)$.

Thus if $\ssaa{g}{h}{b}=0$ for all $h\in G$, then $b$ has infinitely many
roots and hence is $0$.  By Corollary~\ref{S-prime module from prime
module R*G}, this implies the ideal $0$ of $R$ is $S$-prime.  However,
$0$ is not \GGi, since $\sigma(w)^{-1}(0)=tR$.
\end{ex}

\begin{ex}\label{reversibility necessary for invariance auto}
Let $G$ be the free monoid on $v,w$, let $F$ be a field, let
$R=F[\dots,t_{-1},t_0,t_1,\dots]$, and let $I$ be the ideal of $R$
generated by $t_0,t_1,\dots$.

We begin by defining permutations $\alpha,\beta$ of $\bbZ$ by
$\alpha(n)=n+2$ and
\[
\beta(n)=
\begin{cases}n&\text{ if $n\ge0$;}
\\
n-2&\text{ if $n\le-1$ and $n$ is odd;}
\\
-1&\text{ if $n=-2$;}
\\
n+2&\text{ if $n\le-4$ and $n$ is even;}
\end{cases}
\]
In cycle form (with infinite cycles),
$\alpha=(\dots\,{-3}\,{-1}\,1\,3\,\dots)(\dots\,{-2}\,0\,2\,\dots)$ and
$\beta=(\dots\,{-4}\,{-2}\,{-1}\,{-3}\,\dots)$.

Define $\sigma:G\to\Aut_{F-\text{Alg}}(R)$ by
$\ssaa{v}{}{t_n}=t_{\alpha(n)}$ and $\ssaa{w}{}{t_n}=t_{\beta(n)}$.

Clearly $I$ is \GGs{} but not \GGi, since $\ssa{v}{I}\subsetneq I$.  In
addition, $I$ is a prime ideal.  We will show $I$ is $S$-prime.  By
Corollary~\ref{S-prime module from prime module R*G}, it is enough to
show that for any $r\in R, g\in G$, if $\ssaa{g}{h}{r}\in I$ for all $h\in
G$, then $r\in I$.  Because $I$ is a monomial ideal and each $\sigma(g)$
takes distinct monomials to distinct monomials, it is enough to show this in
the case where $r$ is a monomial.

Thus we will assume that $g\in G, r\in R\smin I$, that $r$ is a product of
indeterminates $t_n$ with each $n<0$, and that $\ssaa{g}{h}{r}\in I$ for
every $h\in G$.  We will show below that for any negative integers $a,n$,
there are integers $m,b$ such that $m$ is positive, $b$ is odd,  $b<a$, and
$\ssaa{w^m}{}{t_n}=t_b$. Suppose we have proven this claim and
suppose $g$ has $k$ occurrences of $v$. For each $t_n$ that occurs in
$r$, let $m(n)$ be a positive integer as above with
$\ssaa{w^{m(n)}}{}{t_n}=t_{b(n)}$ where $b(n)<-2k$ is odd. Suppose
$c<0$ is odd.  If we apply $\sigma(w)$ to $t_c$, we get $t_{c-2}$. If we
apply $\sigma(v)$ to $t_c$, we get $t_{c+2}$. In either case, the
subscript is odd and no more than two larger than $c$. Thus when we
apply $\sigma(g)$ to $t_{b(n)}$, the result is $t_{d(n)}$ where $d(n)<0$
is odd. Hence if we let $m$ be the maximum of the $m(n)$, we see that
$\ssaa{g}{w^m}{r}$ is a monomial all of whose indeterminate factors
have negative subscripts. Thus $\ssaa{g}{w^m}{r}\notin I$, which
contradicts our choice of $r$.

To finish the proof, we need to verify the claim.  That is, we need to show
that if $n,a<0$, there exist a positive $m$ and an odd $b<a$ such that
$\beta^m(n)=b$.  If $n$ is odd, we may choose $m=\abs{a}$ and
$b=n-2m$. If $n=2k$ is even, we may choose $m=\abs{k}+\abs{a}$ and
$b=-1-2\abs{a}$.
\end{ex}

\medskip

The next example shows that Corollaries~\ref{S-prime ideal R*G comm
or surj} and~\ref{S-prime module R*G comm or surj} fail in the
commutative case if \revsh{{} does not hold. We construct an example
where $G$ is the free monoid on two generators, acting on a commutative
base ring $R$, and $I$ is a \GGi{} ideal of $R$ that is not $S$-prime, but
is such that whenever $A,B\normal R$, and $B$ is \GGs, if $AB\sub I$,
then either $A\sub I$ or $B\sub I$.

Note that Corollary~\ref{S-prime ideal R*G comm or surj} holds in the
case where \psurj{} without \revsh{} hypothesis.  Unfortunately, we do
not presently have an example of the necessity of \revsh{} in the ``\psurj''
case of Corollary~\ref{S-prime module R*G comm or surj}.  Nor do we
have an example that Corollaries~\ref{S-prime ideal R*G comm or surj}
and Corollary~\ref{S-prime module R*G comm or surj} fail when $R$ is
commutative noetherian.

\begin{ex}\label{reversibility necessary for S-prime comm}
Let $G$ be the free monoid on $v,w$, let $F$ be a field, let
$R=F[t_0,t_1,\dots]$, and let $I$ be the ideal of $R$ generated by all
$t_it_j$ with $i\ne j$.  (We could also allow all $t_j^2\in I$ without
affecting the result.)

We begin by defining almost permutations $\alpha,\beta$ of $\bbN$ by
$\alpha(n)=n+1$ and
\[
\beta(n)=
\begin{cases}1&\text{ if $n=0$;}
\\
n+2&\text{ if $n\ge 1$ and $n$ is odd;}
\\
n-2&\text{ if $n\ge 2$ and $n$ is even.}
\end{cases}
\]
In ``cycle form'' $\alpha=(0\,1\,2\,\dots)$ and
$\beta=(\dots\,{4}\,{2}\,0\,1\,3\,\dots)$.  Of course $\alpha$ is not a
permutation or true cycle; it is injective but not surjective.  The map
$\beta$ is a bijection.

Define $\sigma:G\to\Aut_{F-\text{Alg}}(R)$ by
$\ssaa{v}{}{t_n}=t_{\alpha(n)}$ and $\ssaa{w}{}{t_n}=t_{\beta(n)}$.

We will first show that $I$ is \GGi.  It is obvious that $I$ is \GGs. Suppose
$r$ is a nontrivial monomial and $\ssa{g}{r}\in I$ for some $g\in G$.
Clearly we cannot have $r=t_n^k$.  However, if $r$ is any other nontrivial
monomial, then $r$ must have at least two distinct indeterminate factors,
and thus $r\in I$.  This proves $I$ is \GGi.

Set $g=v$.  Then $\alpha(n)>0$ for all $n\in\bbN$, so $\ssa{g}{r}\in I$
for all indeterminates $r$.  If we set $r=t_0$, we have $\ssaa{g}{h}{r}\in
I$ for all $h\in G$.  Since $r\notin I$, Corollary~\ref{S-prime module
from prime module R*G} shows $I$ is not $S$-prime.
\end{ex}

\bigskip


\end{document}